\documentclass[a4paper,leqno,10pt]{amsart}

\usepackage{epsf,graphicx,epsfig}
\usepackage{amscd}
\usepackage{amsmath,latexsym,amssymb,amsthm}
\usepackage[nospace,noadjust]{cite}
\usepackage{textcomp}
\usepackage{setspace,cite}
\usepackage{lscape,fancyhdr,fancybox}
\usepackage{stmaryrd}
\usepackage[all,cmtip]{xy}
\usepackage{tikz}
\usepackage{cancel}
\usetikzlibrary{shapes,arrows,decorations.markings}
\usepackage{graphicx}

\usepackage{color}
\usepackage{url}
\usepackage{enumerate}
\usepackage[mathscr]{euscript}
  \theoremstyle{plain}

\swapnumbers
    \newtheorem{theorem}{Theorem}[section]
    \newtheorem{proposition}[theorem]{Proposition}
     
   \newtheorem{lemma}[theorem]{Lemma}
    \newtheorem{corollary}[theorem]{Corollary}
    
    \newtheorem{subsec}[theorem]{}
\theoremstyle{definition}
    \newtheorem{definition}[theorem]{Definition}
        \newtheorem{remark}[theorem]{Remark}
\theoremstyle{remark}

\title{}
\author{}
\date{}
\usepackage{amssymb}

\usepackage{hyperref}
\hypersetup{
	colorlinks,
	citecolor=blue,
	filecolor=black,
	linkcolor=blue,
	urlcolor=black
}

\begin{document}
\title{Non-abelian extensions of Rota-Baxter Lie algebras and inducibility of automorphisms}

\author{Apurba Das}
\address{Department of Mathematics,
Indian Institute of Technology Kharagpur, Kharagpur-721302, West Bengal, India.}
\email{apurbadas348@gmail.com, apurbadas348@maths.iitkgp.ac.in}

\author{Samir Kumar Hazra}
\address{Harish-Chandra Research Institute, Chhatnag Road, Jhusi, Prayagraj-211019, India.}
\email{samirhazra@hri.res.in}

\author{Satyendra Kumar Mishra}
\address{Institute for Advancing Intelligence, TCG Centres for Research and Education in Science and Technology, Kolkata, West Bengal, India.}
\email{satyamsr10@gmail.com}

%\begin{center}

%\end{center}
%

%\author{Samir Kumar Hazra}
%\address{\textcolor{red}{LEFT}}
%\email{\textcolor{red}{LEFT}}

%\author{Satyendra Kumar Mishra}
%

%\curraddr{}
%\email{}


\begin{abstract}
A Rota-Baxter Lie algebra $\mathfrak{g}_T$ is a Lie algebra $\mathfrak{g}$ equipped with a Rota-Baxter operator $T : \mathfrak{g} \rightarrow \mathfrak{g}$. In this paper, we consider non-abelian extensions of a Rota-Baxter Lie algebra $\mathfrak{g}_T$ by another Rota-Baxter Lie algebra $\mathfrak{h}_S.$ We define the non-abelian cohomology $H^2_{nab} (\mathfrak{g}_T, \mathfrak{h}_S)$ which classifies {equivalence classes of} such extensions. Given a non-abelian extension
\[
\xymatrix{
0 \ar[r] & \mathfrak{h}_S \ar[r]^{i} & \mathfrak{e}_U \ar[r]^{p} & \mathfrak{g}_T \ar[r] & 0
}
\]
of Rota-Baxter Lie algebras, we also show that the obstruction for a pair of Rota-Baxter automorphisms in $\mathrm{Aut}(\mathfrak{h}_S ) \times \mathrm{Aut}(\mathfrak{g}_T)$ to be induced by an automorphism in $\mathrm{Aut}(\mathfrak{e}_U)$ lies in the cohomology group $H^2_{{nab}} (\mathfrak{g}_T, \mathfrak{h}_S)$. As a byproduct, we obtain the Wells short-exact sequence in the context of Rota-Baxter Lie algebras. Finally, we show how these results fit with abelian extensions of Rota-Baxter Lie algebras.
\end{abstract}

%{17B38, 16D20, 16E40, 18N40}

\maketitle

%\noindent {\bf Abstract:} \\

%\qquad 2020 Mathematics Subject Classification: {17B38, 16D20, 16E40, 18N40.}

\noindent

\thispagestyle{empty}

\tableofcontents

\vspace{0.2cm}

\section{Introduction}

Various kinds of extensions (e.g., central extensions, abelian extensions, non-abelian extensions, etc.) of algebraic structures were first developed by Hochschild, Eilenberg, Maclane, and Serre, among others \cite{hochschild,hoch,serre,lyndon,eilen}. Among all extensions, the non-abelian extension theory is the most general one. In \cite{eilen} Eilenberg and Maclane first considered non-abelian extensions of abstract groups, which led them to introduce low-dimensional non-abelian cohomology groups. Subsequently, similar results for (super) Lie algebras, Leibniz algebras, Lie $2$-algebras, $L_{\infty}$-algebras are also considered in the literature \cite{casas,Chen,fial-pen,hazra-habib,hoch,inas,laza,liu}. A non-abelian extension of a Lie algebra $\mathfrak{g}$ by another Lie algebra $\mathfrak{h}$ is a short exact sequence $0 \rightarrow \mathfrak{h} \rightarrow \mathfrak{e} \rightarrow \mathfrak{g} \rightarrow 0$ of Lie algebras. It has been shown in \cite{fial-pen,hoch,inas} that non-abelian extensions of Lie algebras can be characterized in terms of derivations of Lie algebras. In \cite{yael} the author described non-abelian extensions in terms of the Deligne groupoid of a suitable differential graded Lie algebra. An abelian extension of a Lie algebra $\mathfrak{g}$ by a representation $\mathfrak{h}$ is a particular type of non-abelian extension in which the Lie bracket of $\mathfrak{h}$ is trivial, and the induced representation on $\mathfrak{h}$ is the prescribed one. It is well-known that the Chevalley-Eilenberg cohomology of Lie algebras classifies abelian extensions. 

\medskip

Another interesting study related to extensions of algebraic structures is given by the inducibility of pair of automorphisms.  See \cite{jin,passi,robinson,wells} for more details about this problem in (abelian) group extensions. In the context of Lie algebras, the problem can be described as follows. Let $0 \rightarrow \mathfrak{h} \xrightarrow{i} \mathfrak{e} \xrightarrow{p} \mathfrak{g} \rightarrow 0$ be an (abelian) extension of Lie algebras. Let $\mathrm{Aut}_\mathfrak{h} (\mathfrak{e})$ be the space of all Lie algebra automorphisms $\gamma \in \mathrm{Aut}(\mathfrak{e})$ that satisfies $\gamma|_\mathfrak{h} \subset \mathfrak{h}$. Then there is a group homomorphism $\tau : \mathrm{Aut}_\mathfrak{h} (\mathfrak{e}) \rightarrow \mathrm{Aut}(\mathfrak{h}) \times \mathrm{Aut}(\mathfrak{g})$, $\tau (\gamma) = (\gamma|_\mathfrak{h}, p \gamma s)$, where $s$ is a section of the map $p$. A pair of automorphisms $(\beta, \alpha) \in \mathrm{Aut}_\mathfrak{h} (\mathfrak{e}) \rightarrow \mathrm{Aut}(\mathfrak{h}) \times \mathrm{Aut}(\mathfrak{g})$ is said to be inducible if $(\beta, \alpha)$ lies in the image of $\tau$. The inducibility problem then asks the following: 
When a pair of automorphisms in $\mathrm{Aut}(\mathfrak{h}) \times \mathrm{Aut}(\mathfrak{g})$ is inducible?
The answer to this question is already addressed in \cite{bar-singh,hazra-habib} for abelian extensions. In the same references, the authors also find the analogue of the Wells exact sequence that connects automorphisms, derivations and cohomology of Lie algebras.

\medskip

Algebras are often equipped with additional structures. For instance, classical algebras such as Lie algebras, associative algebras and higher homotopy algebras with involutions appear in the standard constructions of algebras arising from some geometric contexts, where the underlying geometric object has an involution \cite{braun,costello}. Such algebras are usually called $\ast$-algebras or involutive algebras. In \cite{loday-der} Loday considered algebras equipped with distinguished derivation. Their cohomology and deformation theory is considered in \cite{lieder,das-jhrs}. Recently, Rota-Baxter operators on associative algebras pay very much attention due to connections with quantum field theory, algebraic combinatorics, splitting of algebras, Yang-Baxter equations and infinitesimal bialgebras \cite{connes,bai,aguiar}. The notion of the Rota-Baxter operator originated in the work of Baxter in the fluctuation theory of probability \cite{baxter}. Subsequently, such an operator was studied by Rota, among others \cite{rota} and recently by Guo with his coauthors \cite{guo-book}. Rota-Baxter operators on Lie algebras first appeared in the work of Kuperschmidt in the study of classical $r$-matrices \cite{kuper}. A Lie algebra $\mathfrak{g}$ equipped with a Rota-Baxter operator $T: \mathfrak{g} \rightarrow \mathfrak{g}$ is called a Rota-Baxter Lie algebra. We denote a Rota-Baxter algebra as above simply by $\mathfrak{g}_T$. See the next section for more details. Cohomology of Rota-Baxter Lie algebras with coefficients in a representation was introduced, and abelian extensions are studied in \cite{jiang-sheng} (see also \cite{laza-rota,DasSK2}).

\medskip

In this paper, we consider non-abelian extensions of a Rota-Baxter Lie algebra $\mathfrak{g}_T$ by another Rota-Baxter Lie algebra $\mathfrak{h}_S$. We observed that such non-abelian extensions could be described by certain triples of maps satisfying some suitable properties. This allows to construct the non-abelian cohomology $H^2_{nab}(\mathfrak{g}_T, \mathfrak{h}_S)$ which classifies equivalence classes of non-abelian extensions of $\mathfrak{g}_T$ by $\mathfrak{h}_S$. 

\medskip

In the next part, we consider the inducibility of a pair of automorphisms in a non-abelian extension of Rota-Baxter Lie algebras. We first find a necessary and sufficient condition for a pair of automorphisms to be inducible. This result motivates us to define the Wells map $\mathcal{W} : \mathrm{Aut}(\mathfrak{h}_S) \times \mathrm{Aut}(\mathfrak{g}_T) \rightarrow H^2_{nab} (\mathfrak{g}_T, \mathfrak{h}_S)$ in the context of non-abelian extension of Rota-Baxter Lie algebras. In terms of the Wells map, a pair $(\beta, \alpha) \in \mathrm{Aut}(\mathfrak{h}_S) \times \mathrm{Aut}(\mathfrak{g}_T)$ is inducible if and only if $\mathcal{W} ((\beta, \alpha)) = 0$.  We also derive the analogue of the Wells short exact sequence connecting various automorphism groups and the cohomology $H^2_{nab} (\mathfrak{g}_T, \mathfrak{h}_S)$. Finally, we observe how the above results fit with the abelian extensions of a Rota-Baxter Lie algebra by a representation.

\medskip

The paper is organized as follows. In Section \ref{sec-2}, we recall some necessary background on Rota-Baxter Lie algebras. In Section \ref{sec-3}, we consider non-abelian extensions of Rota-Baxter Lie algebras and classify them in terms of non-abelian cohomology. Given a non-abelian extension of Rota-Baxter Lie algebras, we consider the inducibility problem and the Wells short exact sequence in Section \ref{sec-4} and Section \ref{sec-5}, respectively. Finally, in Section \ref{sec-6}, we revise abelian extensions of Rota-Baxter Lie algebras and consider the inducibility problem in abelian extensions.

\medskip

All vector spaces, (multi)linear maps, Lie algebras, and wedge products are over a field ${\bf k}$ of characteristic $0$. We usually denote the elements of the Lie algebra $\mathfrak{g}$ by $x, y, z, x_1, x_2, \ldots$ and the elements of the Lie algebra $\mathfrak{h}$ by $h, k, l, h_1, h_2, \ldots$.

\medskip

\section{Rota-Baxter Lie algebras}\label{sec-2}
In this section, we recall some basics of Rota-Baxter Lie algebras that are required in the course of our study. Our main references are \cite{kuper,jiang-sheng,laza-rota}.

\begin{definition}
Let $\mathfrak{g} = (\mathfrak{g}, [~,~]_\mathfrak{g})$ be a Lie algebra. A {\bf Rota-Baxter operator} on $\mathfrak{g}$ is a linear map $T : \mathfrak{g} \rightarrow \mathfrak{g}$ satisfying
\begin{align*}
[T(x), T(y)]_\mathfrak{g} = T ( [T(x), y]_\mathfrak{g} + [x, T(y)]_\mathfrak{g}), \text{ for } x, y \in \mathfrak{g}.
\end{align*}
\end{definition}

A {\bf Rota-Baxter Lie algebra} is a Lie algebra $\mathfrak{g}$ equipped with a Rota-Baxter operator $T: \mathfrak{g} \rightarrow \mathfrak{g}$. We denote a Rota-Baxter Lie algebra simply by the notation $\mathfrak{g}_T$.

\begin{definition}
Let $\mathfrak{g}_T$ and $\mathfrak{g}'_{T'}$ be two Rota-Baxter Lie algebras. A {\bf morphism} $\varphi: \mathfrak{g}_T \rightarrow \mathfrak{g}'_{T'}$ of Rota-Baxter Lie algebras is given by a Lie algebra homomorphism $\varphi : \mathfrak{g} \rightarrow \mathfrak{g}'$ satisfying ${T}' \circ \varphi = \varphi \circ T$. It is said to be an {\bf isomorphism} if $\varphi$ is a linear isomorphism.
\end{definition}

Let $\mathfrak{g}_T$ be a Rota-Baxter Lie algebra. We denote by $\mathrm{Aut}(\mathfrak{g}_T)$ the set of all automorphisms (self isomorphisms) of the Rota-Baxter Lie algebra $\mathfrak{g}_T$. Then $\mathrm{Aut} (\mathfrak{g}_T)$ has an obvious group structure, called the automorphism group.

\medskip

\begin{definition}
Let $\mathfrak{g}_T$ be a Rota-Baxter Lie algebra. A {\bf representation} of $\mathfrak{g}_T$ consists of a $\mathfrak{g}$-module $\mathfrak{h}$ (i.e., $\mathfrak{h}$ is a vector space with a Lie algebra homomorphism $\psi : \mathfrak{g} \rightarrow \mathrm{End}(\mathfrak{h})$, called the action map) equipped with a linear map $S : \mathfrak{h} \rightarrow \mathfrak{h}$ satisfying
\begin{align*}
\psi_{T(x)}  S(h) = S \big(  \psi_{T(x)}  h + \psi_{ x } S(h) \big), \text{ for } x \in \mathfrak{g}, h \in \mathfrak{h}.
\end{align*}
\end{definition}
A representation as above may be simply denoted by $\mathfrak{h}_S$ when the action map $\psi$ is clear from the context. It is clear that any Rota-Baxter Lie algebra $\mathfrak{g}_T$ is a representation of itself where the $\mathfrak{g}$-module structure on $\mathfrak{g}$ is given by the adjoint action.

\medskip

Let $\mathfrak{g}$ be a Lie algebra and $\mathfrak{h}$ be a $\mathfrak{g}$-module. The Chevalley-Eilenberg cochain complex of $\mathfrak{g}$ with coefficients in $\mathfrak{h}$ is given by $\{ C^\bullet (\mathfrak{g}, \mathfrak{h}), \delta^\bullet \}$, where $C^n (\mathfrak{g}, \mathfrak{h}) = \mathrm{Hom} (\wedge^n \mathfrak{g}, \mathfrak{h})$ for $n \geq 0$, and
\begin{align*}
\delta^n  (f) (x_1, \ldots, x_{n+1}) =~& \sum_{i=1}^{n+1}  (-1)^{i+1} ~ \psi_{x_i} f (x_1, \ldots, \widehat{x_i}, \ldots, x_{n+1}) \\
~&+ \sum_{1 \leq i < j \leq n+1} (-1)^{i+j}~ f ([x_i, x_j]_\mathfrak{g}, x_1, \ldots, \widehat{x_i}, \ldots, \widehat{x_j}, \ldots, x_{n+1}), \text{ for } f \in C^n (\mathfrak{g}, \mathfrak{h}).
\end{align*}
Next, let $\mathfrak{g}_T$ be a Rota-Baxter Lie algebra and $\mathfrak{h}_S$ be a representation of it. Then there is another cochain complex $\{ C^\bullet (\mathfrak{g}, \mathfrak{h}), \partial^\bullet \}$ on the same cochain groups with differential
\begin{align*}
\partial^n (f) (x_1, \ldots, x_{n+1} ) =& \sum_{i=1}^{n+1}  (-1)^{i+1} ~ \big( \psi_{T(x_i)} f (x_1, \ldots, \widehat{x_i}, \ldots, x_{n+1}) - S ( \psi_{x_i} f (x_1, \ldots, \widehat{x_i}, \ldots, x_{n+1}) ) \big) \\
~&+ \sum_{1 \leq i < j \leq n+1} (-1)^{i+j}~ f ([Tx_i, x_j]_\mathfrak{g} + [x_i, Tx_j]_\mathfrak{g}, x_1, \ldots, \widehat{x_i}, \ldots, \widehat{x_j}, \ldots, x_{n+1}),
\end{align*}
for $f \in C^n(\mathfrak{g}, \mathfrak{h}).$ Combining the above two complexes, the authors in \cite{jiang-sheng} introduced a new cochain complex $\{ C^\bullet_\mathrm{RBLie} (\mathfrak{g}_T, \mathfrak{h}_S) , \delta_\mathrm{RBLie} \}$, where
\begin{align*}
C^0_\mathrm{RBLie} (\mathfrak{g}_T, \mathfrak{h}_S) = 0, \quad C^1_\mathrm{RBLie} (\mathfrak{g}_T, \mathfrak{h}_S) = C^1 (\mathfrak{g}, \mathfrak{h}) ~~~ \text{ and } ~~~ C^n_\mathrm{RBLie} (\mathfrak{g}_T, \mathfrak{h}_S) = C^n (\mathfrak{g}, \mathfrak{h}) \oplus C^{n-1} (\mathfrak{g}, \mathfrak{h}), \text{ for } n \geq 2
\end{align*}
and $\delta_\mathrm{RBLie} : C^n_\mathrm{RBLie} (\mathfrak{g}_T, \mathfrak{h}_S) \rightarrow C^{n+1}_\mathrm{RBLie} (\mathfrak{g}_T, \mathfrak{h}_S)$ is given by
\begin{align*}
\delta^n_\mathrm{RBLie} (f, \Theta) = \big( \delta^n (f),~ \partial^{n-1} (\Theta) + (-1)^n f \circ  T^{\otimes n} - (-1)^n \sum_{i=1}^n Sf \circ (T \otimes \cdots \otimes \underbrace{\mathrm{id}}_{i\text{-th}} \otimes \cdots \otimes T)  \big),
\end{align*}
for $(f, \Theta) \in C^n_\mathrm{RBLie} (\mathfrak{g}_T, \mathfrak{h}_S)$. The cohomology of the cochain complex $\{ C^\bullet_\mathrm{RBLie} (\mathfrak{g}_T, \mathfrak{h}_S) , \delta_\mathrm{RBLie} \}$ is called the {\bf cohomology} of the Rota-Baxter Lie algebra $\mathfrak{g}_T$ with coefficients in $\mathfrak{h}_S$. This cohomology will be used only in Section \ref{sec-6}, where we revisit abelian extensions of Rota-Baxter Lie algebras from \cite{jiang-sheng}. 

A linear map $d : \mathfrak{g} \rightarrow \mathfrak{h}$ is called a derivation on $\mathfrak{g}_T$ with values in $\mathfrak{h}_S$ if $\delta^1_\mathrm{RBLie} (d) = 0$, equivalently,
\begin{align*}
d [x,y]_\mathfrak{g} = \psi_x dy - \psi_y dx ~~~ \text{ and } ~~~ S \circ d = d \circ T, ~ \text{ for } x, y \in \mathfrak{g}.
\end{align*}
We denote the set of all derivations on $\mathfrak{g}_T$ with values in $\mathfrak{h}_S$ by $\mathrm{Der}(\mathfrak{g}_T, \mathfrak{h}_S)$.

\section{Non-abelian Extensions of Rota-Baxter Lie algebras}\label{sec-3}

In this section, we study non-abelian extensions of a Rota-Baxter Lie algebra by another Rota-Baxter Lie algebra. We define the second non-abelian cohomology space that classifies equivalence classes of such extensions.

\begin{definition}
Let $\mathfrak{g}_T$ and $\mathfrak{h}_S$ be two Rota-Baxter Lie algebras. A non-abelian {\bf extension} of $\mathfrak{g}_T$ by $\mathfrak{h}_S$ is a Rota-Baxter Lie algebra $\mathfrak{e}_U$ equipped with a  short exact sequence of Rota-Baxter Lie algebras
\begin{align}\label{abelian-diagram}
\xymatrix{
0 \ar[r] & \mathfrak{h}_S \ar[r]^{i} & \mathfrak{e}_U \ar[r]^{p} & \mathfrak{g}_T \ar[r] & 0.
}
\end{align} 
We often denote a non-abelian extension as above simply by $\mathfrak{e}_U$ when the underlying short exact sequence is clear from the context.
\end{definition}

\begin{definition}\label{defn-nab-equiv}
Let $\mathfrak{e}_U$ and $\mathfrak{e}'_{U'}$ be two non-abelian extensions of $\mathfrak{g}_T$ by $\mathfrak{h}_S$. They are said to be {\bf equivalent} if there is a morphism $\varphi: \mathfrak{e}_U \rightarrow \mathfrak{e}'_{U'}$ of Rota-Baxter Lie algebras making the following diagram commutative
\begin{align}\label{abelian-equiv}
%\xymatrixrowsep{0.36cm}
%\xymatrixcolsep{0.36cm}
\xymatrix{
0 \ar[r] & \mathfrak{h}_S  \ar@{=}[d] \ar[r]^{i} & \mathfrak{e}_U \ar[d]^\varphi \ar[r]^{p} & \mathfrak{g}_T  \ar@{=}[d] \ar[r] & 0 \\
0 \ar[r] & \mathfrak{h}_S \ar[r]_{i'} & \mathfrak{e}'_{U'} \ar[r]_{p'} & \mathfrak{g}_T \ar[r] & 0.
}
\end{align}
We denote by $\mathrm{Ext}_{nab}(\mathfrak{g}_T, \mathfrak{h}_S)$ the set of all equivalence classes of non-abelian extensions of $\mathfrak{g}_T$ by $\mathfrak{h}_S$. 
\end{definition}

\medskip

\medskip

Let $\mathfrak{e}_U$ be a non-abelian extension of the Rota-Baxter Lie algebra $\mathfrak{g}_T$ by $\mathfrak{h}_S$ as of (\ref{abelian-diagram}). A section of $p$ is a linear map $s: \mathfrak{g} \rightarrow \mathfrak{e}$ that satisfies $p \circ s = \mathrm{id}_\mathfrak{g}$. A section of $p$ always exists. Given a section $s : \mathfrak{g} \rightarrow \mathfrak{e}$, we define maps
%\begin{align*}
$\chi : \wedge^2 \mathfrak{g} \rightarrow \mathfrak{h}$, $\psi : \mathfrak{g} \rightarrow \mathrm{Der}(\mathfrak{h})$ and $\Phi : \mathfrak{g} \rightarrow \mathfrak{h}$ by 
%\end{align*}
\begin{align}\label{three-maps}
\begin{cases} \chi(x,y):=[s(x),s(y)]_{\mathfrak{e}}-s[x,y]_{\mathfrak{g}},\\
\psi_x(h):=[s(x),h]_{\mathfrak{e}},\\
\Phi(x):= U(s(x))-s(T(x)), \text{ for } x,y \in \mathfrak{g} \text{ and } h \in \mathfrak{h}.
\end{cases}
\end{align}
It has been observed in \cite{yael} that the maps $\chi$ and $\psi$ satisfy the following
\begin{align}
&\psi_{x}\psi_{y}(h)-\psi_{y}\psi_{x}(h)-\psi_{[x,y]_\mathfrak{g}}(h)=[\chi(x,y),h]_{\mathfrak{h}}, \label{cocycle-condition I}\\
&\psi_{x}\chi(y,z)+\psi_{y}\chi(z,x)+\psi_{z}\chi(x,y)-\chi([x,y]_\mathfrak{g},z)-\chi([y,z]_\mathfrak{g},x) - \chi([z,x]_\mathfrak{g},y)=0,\label{cocycle-condition II}
\end{align}
for $x,y,z \in \mathfrak{g}$ and $h \in \mathfrak{h}$. We claim that the above maps additionally satisfy 
\begin{align}
&\psi_{T(x)}S(h)=S\big(\psi_x S(h)+\psi_{T(x)}h \big)+S[\Phi(x),h]_{\mathfrak{h}}-[\Phi(x),S(h)]_{\mathfrak{h}}, \label{cocycle-condition III}\\
&\chi(T(x),T(y))-S\big(\chi(x,T(y))+\chi(T(x),y)\big)-\Phi\big([x, T(y)]_{\mathfrak{g}}+[T(x),y]_{\mathfrak{g}}\big) \label{cocycle-condition IV}\\
&\qquad +\psi_{T(x)}\Phi(y)-\psi_{T(y)}\Phi(x)-S\big(\psi_{x}\Phi(y)-\psi_{y}\Phi(x)\big)+[\Phi(x),\Phi(y)]_{\mathfrak{h}}=0, \nonumber
\end{align}
for $x, y \in \mathfrak{g}$ and $h \in \mathfrak{h}$. We first observe that
\begin{align*}
&\psi_{T(x)}S(h)-S\big(\psi_x S(h)+\psi_{T(x)}h\big)-S[\Phi(x),h]_{\mathfrak{h}}+[\Phi(x),S(h)]_{\mathfrak{h}}\\
&=\cancel{[sT(x),S(h)]_\mathfrak{e}}-S\big([s(x),S(h)]_\mathfrak{e} +[sT(x),h]_\mathfrak{e} \big)-S\big([Us(x),h]_\mathfrak{e} - [sT(x),h]_\mathfrak{e} \big)\\
& ~~~ +[Us(x),S(h)]_\mathfrak{e} -\cancel{[sT(x),S(h)]_\mathfrak{e}}\\
&=-U\big([s(x), U(h)]_\mathfrak{e} +\cancel{[sT(x),h]_\mathfrak{e}}\big)- U \big([Us(x),h]_\mathfrak{e} -\cancel{[sT(x),h]_\mathfrak{e}}\big)+[Us(x),U(h)]_\mathfrak{e} \quad (\mbox{since } U|_{\mathfrak{h}}=S)\\
&=[Us(x),U(h)]_\mathfrak{e} - U[s(x),U(h)]_\mathfrak{e} -U[Us(x),h]_\mathfrak{e} =~0 \qquad (\mbox{as }U \mbox{ is a Rota-Baxter operator}).
\end{align*}
Similarly, by a direct calculation we get
\begin{align*}
&\chi(T(x),T(y))-S\big(\chi(x,T(y))+\chi(T(x),y)\big)-\Phi\big([x, T(y)]_{\mathfrak{g}}+[T(x),y]_{\mathfrak{g}}\big) \nonumber \\\nonumber
& ~~~~ +\psi_{T(x)}\Phi(y)-\psi_{T(y)}\Phi(x)-S\big(\psi_{x}\Phi(y)-\psi_{y}\Phi(x)\big)+[\Phi(x),\Phi(y)]_{\mathfrak{h}} \nonumber \\
&=\cancel{[sT(x),sT(y)]_\mathfrak{e}}-\underbrace{s[T(x),T(y)]_\mathfrak{g}}_{(1A)}
-S\big(\underbrace{[s(x),sT(y)]_\mathfrak{e}}_{(2B)}-\underbrace{s[x,T(y)]_\mathfrak{g}}_{(2C)}\big)-S\big(\underbrace{[sT(x),s(y)]_\mathfrak{e}}_{(2A)}-\underbrace{s([T(x),y]_\mathfrak{g})}_{(2D)}\big) \nonumber \\
& ~~~~ -\underbrace{Us([x,T(y)]_\mathfrak{g})}_{(2C)}+\underbrace{sT([x,T(y)]_\mathfrak{g})}_{(1A)}-\underbrace{Us([T(x),y]_\mathfrak{g})}_{(2D)}+\underbrace{sT([T(x),y]_\mathfrak{g})}_{(1A)}+\cancel{[sT(x), Us(y)]_\mathfrak{e}} -\cancel{[sT(x),sT(y)]_\mathfrak{e}} \nonumber \\
& ~~~~ -\cancel{[sT(y), Us(x)]_\mathfrak{e}}+\cancel{[sT(y),sT(x)]_\mathfrak{e}}+S\big(\underbrace{[s(x),sT(y)]_\mathfrak{e}}_{(2B)}-\underbrace{[s(x), Us(y)]_\mathfrak{e}}_{(1B)}\big) \nonumber \\
& ~~~~ +S\big(\underbrace{[s(y), Us(x)]_\mathfrak{e}}_{(1B)}-\underbrace{[s(y),
sT(x)]_\mathfrak{e}}_{(2A)}\big)+\underbrace{[Us(x),Us(y)]_\mathfrak{e}}_{(1B)}-\cancel{[sT(x),Us(y)]_\mathfrak{e}}+\cancel{[sT(x),sT(y)]_\mathfrak{e}}-\cancel{[Us(x),sT(y)]_\mathfrak{e}}. \nonumber
\end{align*}
To see this entire expression vanishes, we observe the followings. The terms underlined with $(1A)$ cancelled as $T$ is a Rota-Baxter operator. On the other hand, since $S = U|_\mathfrak{h}$, the terms underlined with $(1B)$ can be written as 
$$[Us(x), Us(y)]_\mathfrak{e} - U \big([ Us(x),s(y)]_\mathfrak{e} + [s(x),Us(y)]_\mathfrak{e} \big)$$
which vanishes as $U$ is a Rota-Baxter operator on $\mathfrak{e}$. Finally, both the terms in each of the underlined expressions $(2A)$, $(2B)$, $(2C)$, and $(2D)$ cancel each other as $S = U|_\mathfrak{h}$. Thus, we prove our claims.

\medskip

\medskip

Note that the maps $\chi, \psi$ and $\Phi$ (defined in (\ref{three-maps})) depends on the choice of the section $s$. Let $s' : \mathfrak{g} \rightarrow \mathfrak{e}$ be any other section with the corresponding maps $\chi', \psi'$ and $\Phi'$. We define a map $\varphi : \mathfrak{g} \rightarrow \mathfrak{h}$ by $\varphi (x) := s (x) - s'(x)$, for $x \in \mathfrak{g}$. Then, for any $x,y \in \mathfrak{g}$ and $h \in \mathfrak{h}$, we have
\begin{align*}
\psi_{x}(h)-\psi^\prime_{x}(h)&=[s(x),h]_{\mathfrak{e}}-[s^\prime(x),h]_{\mathfrak{e}}\\
&=[\varphi(x),h]_{\mathfrak{h}}, \\
\chi(x,y)-\chi^{\prime}(x,y)&=([s(x),s(y)]_{\mathfrak{e}}-s[x,y]_{\mathfrak{g}})-([s^\prime(x),s^\prime(y)]_{\mathfrak{e}}-s^\prime[x,y]_{\mathfrak{g}})\\
&=([(\varphi+s^{\prime})(x),(\varphi+s^{\prime})(y)]_{\mathfrak{e}}-(\varphi+s^{\prime})[x,y]_{\mathfrak{g}})-([s^\prime(x),s^\prime(y)]_{\mathfrak{e}}-s^\prime[x,y]_{\mathfrak{g}})\\
&=\psi^\prime_{x}(\varphi(y))-\psi^\prime_{y}(\varphi(x))-\varphi([x,y]_{\mathfrak{g}})+[\varphi(x),\varphi(y)]_{\mathfrak{h}},\\
\Phi(x)-\Phi^\prime(x)&= Us(x)-sT(x)- U s^\prime(x)+ s^\prime T(x)\\
&=S(\varphi(x))-\varphi(T(x)).
\end{align*}

\medskip

Motivated by the above discussion, we now define the following definitions.

\begin{definition}
Let $\mathfrak{g}_T$ and $\mathfrak{h}_S$ be two Rota-Baxter Lie algebras. 

\medskip

$\bullet$ A {\bf non-abelian $2$-cocycle} of $\mathfrak{g}_T$ with values in $\mathfrak{h}_S$ is a triple $(\chi,\psi,\Phi)$ of linear maps $\chi: \wedge^2 \mathfrak{g}\rightarrow \mathfrak{h},$ $\psi:\mathfrak{g}\rightarrow \mathrm{Der}(\mathfrak{h})$ and $\Phi: \mathfrak{g}\rightarrow \mathfrak{h}$ satisfying the conditions (\ref{cocycle-condition I}), (\ref{cocycle-condition II}), (\ref{cocycle-condition III}) and (\ref{cocycle-condition IV}).

\medskip

$\bullet$ Let $(\chi,\psi,\Phi)$ and $(\chi',\psi',\Phi')$ be two non-abelian $2$-cocycles of the Rota-Baxter Lie algebra $\mathfrak{g}_T$ with values in $\mathfrak{h}_S$. 
They are said to be \textbf{equivalent} if there exists a linear map $\varphi:\mathfrak{g}\rightarrow\mathfrak{h}$ such that for any $x,y \in \mathfrak{g}$ and $h \in \mathfrak{h},$ the followings are hold:
\begin{align}
\psi_{x}(h)-\psi'_{x}(h)=&~[\varphi(x),h]_{\mathfrak{h}}, \label{equivalent I}\\
\chi(x,y)-\chi^{\prime}(x,y)=&~\psi^{\prime}_{x}(\varphi(y))-\psi^{\prime}_{y}(\varphi(x))-\varphi([x,y]_{\mathfrak{g}})+[\varphi(x),\varphi(y)]_{\mathfrak{h}}, \label{equivalent II}\\
\Phi(x)-\Phi^{\prime}(x)=&~S(\varphi(x))-\varphi(T(x)). \label{equivalent III}
\end{align}
\end{definition}

We denote the set of equivalence classes of non-abelian $2$-cocycles by $H^2_{nab}(\mathfrak{g}_T;\mathfrak{h}_S)$.
With the above notations, we get the following result.

\begin{proposition}\label{prop-ext-nab}
Let $\mathfrak{g}_T$ and $\mathfrak{h}_S$ be two Rota-Baxter Lie algebras. Then there is a well-defined map
\begin{align*}
\Upsilon : \mathrm{Ext}_{nab}(\mathfrak{g}_T, \mathfrak{h}_S) \longrightarrow H^2_{nab} (\mathfrak{g}_T, \mathfrak{h}_S).
\end{align*}
\end{proposition}

\begin{proof}
Let $\mathfrak{e}_U$ and $\mathfrak{e}'_{U'}$ be two equivalent extensions of $\mathfrak{g}_T$ by $\mathfrak{h}_S$ (see Definition \ref{defn-nab-equiv}). Let $s: \mathfrak{g} \rightarrow \mathfrak{e}$ be a section of the map $p$. Then we have $p' \circ (\varphi \circ s) = (p' \circ \varphi) \circ s = p \circ s = \mathrm{id}_\mathfrak{g}$ which shows that $s' = \varphi \circ s : \mathfrak{g} \rightarrow \mathfrak{e}'$ is a section of the map $p'$. Let $(\chi' , \psi', \Phi')$ be the non-abelian $2$-cocycle corresponding to the extension $\mathfrak{e}'_{U'}$ and section $p'$. Then for any $x, y \in \mathfrak{g}$ and $h \in \mathfrak{h}$, we have
\begin{align*}
\chi' (x,y) =~& [s' (x), s'(y)]_{\mathfrak{e}'} - s' [x,y]_\mathfrak{g} \\
=~& [\varphi \circ s (x), \varphi \circ s (y) ]_{\mathfrak{e}'} - \varphi \circ s [x, y]_\mathfrak{g} \\
=~& \varphi \big(   [s (x), s (y) ]_{\mathfrak{e}} - s [x,y]_\mathfrak{g}  \big) \\
=~& \chi (x,y ) \quad (\text{as } \varphi|_\mathfrak{h} = \mathrm{id}_\mathfrak{h}),
\end{align*}
\begin{align*}
\psi_x' (h) = [s'(x) , h]_{\mathfrak{e}'} =~& [\varphi \circ s(x) , h]_{\mathfrak{e}'} \\
=~& [\varphi (s(x)), \varphi (h)]_{\mathfrak{e}'} \\
=~& \varphi ([s(x), h]_\mathfrak{e}) = \varphi \big( \psi_x (h) \big) = \psi_x (h)
\end{align*}
and
\begin{align*}
\Phi'(x) =~& U' (s'(x)) - s' (T(x)) \\
=~& U' (\varphi \circ s (x)) - \varphi \circ s (T(x)) \\
=~& \varphi \big(  U (s(x)) - s (T(x))  \big) = \varphi (\Phi (x)) = \Phi (x).
\end{align*}
This shows that $(\chi, \psi, \Phi) = (\chi', \psi', \Phi')$. Hence their equivalence class in $H^2_{nab} (\mathfrak{g}_T, \mathfrak{h}_S)$ are the same. This shows that the map $\Upsilon : \mathrm{Ext}_{nab} (\mathfrak{g}_T, \mathfrak{h}_S) \rightarrow H^2_{nab} (\mathfrak{g}_T, \mathfrak{h}_S)$ assigning an equivalence class of extensions to the class of corresponding non-abelian $2$-cocycle is well-defined.
\end{proof}

\begin{proposition}\label{prop-nab-ext}
Let $\mathfrak{g}_T$ and $\mathfrak{h}_S$ be two Rota-Baxter Lie algebras. There is a well-defined map
\begin{align*}
\Omega : H^2_{nab} (\mathfrak{g}_T, \mathfrak{h}_S) \longrightarrow \mathrm{Ext}_{nab}(\mathfrak{g}_T, \mathfrak{h}_S).
\end{align*}
\end{proposition}

\begin{proof}
Let $(\chi, \psi, \Phi)$ be a non-abelian $2$-cocycle. Consider the vector space $\mathfrak{g} \oplus \mathfrak{h}$ together with the bilinear skew-symmetric bracket
\begin{align*}
[(x,h),(y,k)]_{\chi,\psi}:= \big( [x,y]_{\mathfrak{g}}, \psi_{x}(k)-\psi_{y}(h)+\chi(x,y)+[h,k]_{\mathfrak{h}} \big)
\end{align*}
and the linear map $U_\Phi : \mathfrak{g} \oplus \mathfrak{h} \rightarrow \mathfrak{g} \oplus \mathfrak{h}$ given by
\begin{align*}
U_\Phi ((x , h)) := (T(x) , S(h) + \Phi (x)),
\end{align*}
for $(x,h), (y,k) \in \mathfrak{g} \oplus \mathfrak{h}.$
Using the conditions (\ref{cocycle-condition I}) and (\ref{cocycle-condition II}), it can be easily verify that the bracket $[~,~]_{\chi, \psi}$ satisfies the Jacobi identity. (We denote this Lie algebra by $\mathfrak{g} \oplus_{\chi, \psi} \mathfrak{h}$). Moreover, we observe that
\begin{align*}
&[U_\Phi ((x , h)) , U_\Phi ((y , k)) ]_{\chi, \psi} \\
&= \big[ \big( T(x), S(h) + \Phi (x) \big), \big(T(y), S(k) + \Phi (y) \big) \big]_{\chi, \psi} \\
&= \big(  [T(x), T(y)]_\mathfrak{g}, ~ \underbrace{\psi_{T(x)} S(k)}_{(A1)} + \underbrace{\psi_{T(x)} \Phi (y)}_{(C1)}  \underbrace{- \psi_{T(y)} S(h)}_{(B1)}  \underbrace{- \psi_{T(y)} \Phi (x)}_{(C2)} + \underbrace{\chi (T(x), T(y))}_{(C3)} \\
& \qquad \qquad \qquad \qquad  \quad + [S(h), S(k)]_\mathfrak{h} + \underbrace{[S(h), \Phi (y)]_\mathfrak{h}}_{(B2)} + \underbrace{[\Phi(x), S(k)]_\mathfrak{h}}_{(A2)} + \underbrace{[\Phi (x), \Phi (y)]_\mathfrak{h}}_{(C4)} \big)\\
&= \bigg(   T[T(x),y]_\mathfrak{g} + T[x, T(y)]_\mathfrak{g} , ~ \underbrace{S (\psi_{x} S(k) + \psi_{T(x)} k) + S [\Phi (x), k]_\mathfrak{h}}_{= (A1) + (A2)} \\
& \quad \underbrace{- S (\psi_{y} S(h) + \psi_{T(y)} h) - S [\Phi (y), h]_\mathfrak{h}}_{=(B1) + (B2)} + S [S(h), k]_\mathfrak{h} + S [h, S(k)]_\mathfrak{h} \\
& \quad \underbrace{- S \big( \psi_{x} \Phi (y) - \psi_{y} \Phi (x) \big) + S \big( \chi (T(x), y) +  \chi (x, T(y) ) \big) + \Phi ([T(x), y]_\mathfrak{g} + [x, T(y)]_\mathfrak{g} }_{= (C1)+(C2)+(C3)+(C4)}\bigg)
\end{align*}

\begin{align*}
&= \bigg( T[T(x),y]_\mathfrak{g} ,~ S \big( \psi_{T(x)} k - \psi_{y} S(h) - \psi_{y} \Phi (x) + \chi (T(x), y) + [S(h) + \Phi (x), k]_\mathfrak{h})  \big) + \Phi [T(x), y]_\mathfrak{g}   \bigg) \\
& \quad + \bigg( T[x, T(y)]_\mathfrak{g}, ~ S \big( \psi_{x} S(k) + \psi_{x} \Phi (y) - \psi_{T(y)} h + \chi (x, T(y)) + [h, S (k)+ \Phi(y)]_\mathfrak{h} \big) + \Phi [x, T(y)]_\mathfrak{g} \bigg) \\   
&=  U_\Phi \big( [T(x),y]_\mathfrak{g}, \psi_{T(x)} k - \psi_{y} S(h) - \psi_{y} \Phi (x) + \chi (T(x), y) + [S(h) + \Phi (x), k]_\mathfrak{h})   \big) \\
& \quad + U_\Phi \big( [x, T(y)]_\mathfrak{g},  \psi_{x} S(k) + \psi_{x} \Phi (y) - \psi_{T(y)} h_ + \chi (x, T(y)) + [h, S (k)+ \Phi(y)]_\mathfrak{h}  \big) \\
&= U_\Phi \big( [ (T(x), S(h) + \Phi(x)), (y, k)]_{\chi, \psi} \big) ~+~ U_\Phi \big( [(x, h), (T(y), S(k) + \Phi (y))]_{\chi, \psi}   \big) \\
&= U_\Phi \big(  [U_\Phi ((x, h)), (y, k)]_{\chi, \psi} + [(x, h), U_\Phi ((y,k))]_{\chi, \psi}  \big).
\end{align*}
This shows that $U_\Phi$ is a Rota-Baxter operator on the Lie algebra $\mathfrak{g} \oplus_{\chi, \psi} \mathfrak{h}$. In other words, $(\mathfrak{g} \oplus_{\chi, \psi} \mathfrak{h})_{U_\Phi}$ is a Rota-Baxter Lie algebra. This is a non-abelian extension of $\mathfrak{g}_T$ by $\mathfrak{h}_S$ with the inclusion $i: \mathfrak{h}_S \rightarrow (\mathfrak{g} \oplus_{\chi, \psi} \mathfrak{h})_{U_\Phi}$ and projection $p : (\mathfrak{g} \oplus_{\chi, \psi} \mathfrak{h})_{U_\Phi} \rightarrow \mathfrak{g}_T$ as structure maps.

\medskip

Let $(\chi, \psi, \Phi)$ and $(\chi', \psi', \Phi')$ be two equivalent $2$-cocycles. Thus, there exists a linear map $\varphi : \mathfrak{g} \rightarrow \mathfrak{h}$ such that the identities (\ref{equivalent I}), (\ref{equivalent II}), (\ref{equivalent III}) hold. Let $(\mathfrak{g} \oplus_{\chi, \psi} \mathfrak{h})_{U_\Phi}$ and $(\mathfrak{g} \oplus_{\chi', \psi'} \mathfrak{h})_{U_{\Phi'}}$ be the Rota-Baxter Lie algebras induced by the $2$-cocycles $(\chi, \psi, \Phi)$ and $(\chi', \psi', \Phi')$, respectively. We define a map
\begin{align*}
\Theta : \mathfrak{g} \oplus \mathfrak{h} \rightarrow \mathfrak{g} \oplus \mathfrak{h} ~\text{ by } \Theta ((x,h)):= (x , \varphi (x) + h),~\text{for } (x, h) \in \mathfrak{g} \oplus \mathfrak{h}.
\end{align*}
Then we have
\begin{align*}
&\Theta [ (x , h), (y , k) ]_{\chi, \psi} \\
&= \Theta \big(  [x, y]_\mathfrak{g}, ~ \psi_{x} k - \psi_{y} h + \chi (x,y) + [h,k]_\mathfrak{h}  \big) \\
&= \big( [x,y]_\mathfrak{g}, \varphi [x,y]_\mathfrak{g} + \psi_{x} k - \psi_{y} h + \chi (x,y) + [h,k]_\mathfrak{h}  \big) \\
&= \big(   [x,y]_\mathfrak{g}, ~ \cancel{\varphi [x,y]_\mathfrak{g}} + \psi'_{x} k + [\varphi(x), k]_\mathfrak{h} - \psi'_{y} h - [\varphi (y), h]_\mathfrak{h} \\
& \quad + \chi' (x,y) + \psi'_{x} \varphi (y) - \psi'_{y} \varphi (x) - \cancel{\varphi [x,y]_\mathfrak{g}} + [\varphi (x), \varphi(y)]_\mathfrak{h} + [h,k]_\mathfrak{h} \big) \\
&= \big( [x,y]_\mathfrak{g}, ~ \psi'_{x} \varphi (y)    + \psi'_{x} k - \psi'_{y} \varphi (x) - \psi'_{y} h + \chi' (x, y) + [\varphi (x) + h, \varphi (y) + k ]_\mathfrak{h} \big) \\
& \qquad \quad (\text{after rearranging}) \\
&= [ ( x, \varphi(x) + h),( y, \varphi(y) + k )]_{\chi', \psi'} \\
&= [\Theta (x,h), \Theta (y,k)]_{\chi', \psi'}
\end{align*}
and
\begin{align*}
(U_{\Phi'} \circ \Theta)((x, h)) 
&= U_{\Phi'} ((x, \varphi (x) + h )) \\
&= (T(x) , S (\varphi (x)) + S(h) + \Phi' (x) )\\
&= (T(x) , \varphi (T(x)) + S (h) + \Phi (x) ) \qquad (\text{from } (\ref{equivalent III}))\\
&= \Theta (T(x), S(h) + \Phi (x)) = (\Theta \circ U_\Phi)((x,h)).
\end{align*}
This shows that the map $\Theta$ defines an equivalence of non-abelian extensions $(\mathfrak{g} \oplus_{\chi, \psi} \mathfrak{h})_{U_\Phi}$ and $(\mathfrak{g} \oplus_{\chi', \psi'} \mathfrak{h})_{U_{\Phi'}}$. Hence there is a well-defined map $\Omega : H^2_{nab} (\mathfrak{g}_T, \mathfrak{h}_S) \rightarrow \mathrm{Ext}_{nab} (\mathfrak{g}_T, \mathfrak{h}_S)$.
\end{proof}

It is straightforward but tedious to verify that the maps $\Upsilon$ and $\Omega$ are inverses to each other. 
As a consequence, we obtain the following classification result of non-abelian extensions.

\begin{theorem}
Let $\mathfrak{g}_T$ and $\mathfrak{h}_S$ be two Rota-Baxter Lie algebras. Then the set of equivalence classes of non-abelian extensions of $\mathfrak{g}_T$ by $\mathfrak{h}_S$ are classified by $H^2_{nab} (\mathfrak{g}_T, \mathfrak{h}_S)$. In other words,
\begin{align*}
\mathrm{Ext}_{nab} (\mathfrak{g}_T, \mathfrak{h}_S) ~\cong~ H^2_{nab} (\mathfrak{g}_T, \mathfrak{h}_S).
\end{align*}
\end{theorem}

\medskip

\section{Inducibility of pair of automorphisms}\label{sec-4}
In this section, we consider the inducibility of a pair of Rota-Baxter automorphisms in a non-abelian extension of Rota-Baxter Lie algebras. Our main results (Theorem \ref{thm-inducibility} and Theorem \ref{thm-inducibility-2}) provide necessary and sufficient conditions for a pair of Rota-Baxter automorphisms to be inducible.

\medskip

Let $0 \rightarrow \mathfrak{h}_S \xrightarrow{i} \mathfrak{e}_U \xrightarrow{p} \mathfrak{g}_T \rightarrow 0$ be a non-abelian extension of Rota-Baxter Lie algebras. Let $\mathrm{Aut}_\mathfrak{h} (\mathfrak{e}_U)$ be the set of all Rota-Baxter automorphisms $\gamma \in \mathrm{Aut}(\mathfrak{e}_U)$ that satisfies $\gamma|_\mathfrak{h} \subset \mathfrak{h}$. For any $\gamma \in \mathrm{Aut}_\mathfrak{h} (\mathfrak{e}_U)$, it follows that $\gamma|_\mathfrak{h} \in \mathrm{Aut}(\mathfrak{h}_S)$. We can also define a map $\overline{\gamma} : \mathfrak{g} \rightarrow \mathfrak{g}$ by
\begin{align*}
\overline{\gamma} (x) := p \gamma s (x), \text{ for } x \in \mathfrak{g}.
\end{align*}
One can easily show that $\overline{\gamma}$ is independent of the choice of $s$.
Note that the space $\mathfrak{g}$ can be regarded as a subspace of $\mathfrak{e}$ via the section $s$. In fact, $\mathfrak{e}$ is isomorphic to $\mathfrak{h} \oplus s (\mathfrak{g})$. With this notation, the map $p$ is simply the projection onto the subspace $\mathfrak{g}$. Since $\gamma$ is an automorphism on $\mathfrak{e}$ preserving the space $\mathfrak{h}$, it also preserves the subspace $\mathfrak{g}$. Therefore, the map $\overline{\gamma} = p \gamma s$ is bijective on $\mathfrak{g}$.

For any $x,y \in \mathfrak{g}$, we have
\begin{align*}
\overline{\gamma} ([x,y]_\mathfrak{g}) = p \gamma (s [x,y]_\mathfrak{g}) =~& p \gamma ( [s(x), s(y)]_\mathfrak{e} - \chi (x,y) ) \\
=~& p \gamma ( [s(x), s(y)]_\mathfrak{e} ) \quad (\text{as } \gamma|_\mathfrak{h} \subset \mathfrak{h} \text{ and } p|_\mathfrak{h} = 0) \\
=~& [p \gamma s (x), p \gamma s (y)]_\mathfrak{g}  = [\overline{\gamma} (x), \overline{\gamma} (y)]_\mathfrak{g}
\end{align*}
and
\begin{align*}
(T \overline{\gamma} - \overline{\gamma} T)(x) =~& (Tp \gamma s - p \gamma s T)(x) \\
=~& (p U \gamma s - p \gamma s T)(x) \quad (\text{as } Tp = pU) \\
=~& p \gamma (Us - sT)(x) \quad (\text{as } U \gamma = \gamma U) \\
=~& 0 \quad (\text{as } (Us -sT)(x) \in \mathfrak{h}, ~ \gamma|_\mathfrak{h} \subset \mathfrak{h} \text{ and } p|_\mathfrak{h} = 0).
\end{align*}
This shows that $\overline{\gamma} : \mathfrak{g}_T \rightarrow \mathfrak{g}_T$ is a Rota-Baxter automorphism, that is, $\overline{\gamma} \in \mathrm{Aut} (\mathfrak{g}_T)$. This construction gives rise to a group homomorphism
\begin{align*}
\tau : \mathrm{Aut}_\mathfrak{h} (\mathfrak{e}_U) \rightarrow \mathrm{Aut}(\mathfrak{h}_S) \times \mathrm{Aut}(\mathfrak{g}_T), ~~ \tau (\gamma) := (\gamma|_\mathfrak{h}, \overline{ \gamma}).
\end{align*}
A pair of Rota-Baxter automorphisms $(\beta, \alpha) \in \mathrm{Aut}(\mathfrak{h}_S) \times \mathrm{Aut}(\mathfrak{g}_T)$ is said to be {\bf inducible} if the pair $(\beta, \alpha)$ lies in the image of $\tau$. 

\medskip

It is then natural to ask the following question:

\medskip

\medskip

\noindent {\bf Question A.} When a pair of Rota-Baxter automorphisms $(\beta, \alpha) \in \mathrm{Aut}(\mathfrak{h}_S) \times \mathrm{Aut}(\mathfrak{g}_T)$ is inducible?

\medskip

\medskip

In the following result, we find a necessary and sufficient condition to answer this question.

\begin{theorem}\label{thm-inducibility}
Let $0 \rightarrow \mathfrak{h}_S \xrightarrow{i} \mathfrak{e}_U \xrightarrow{p} \mathfrak{g}_T \rightarrow 0$ be a non-abelian extension of Rota-Baxter Lie algebras and $(\chi, \psi, \Phi)$ be the corresponding non-abelian $2$-cocycle induced by a section $s$. A pair $(\beta, \alpha) \in \mathrm{Aut}(\mathfrak{h}_S) \times \mathrm{Aut}(\mathfrak{g}_T)$ is inducible if and only if there exists a linear map $\lambda \in \mathrm{Hom}(\mathfrak{g}, \mathfrak{h})$ satisfying the followings:
\begin{itemize}
\item[(I)] $\beta (\psi_x h) - \psi_{\alpha (x)} \beta (h) = [\lambda (x), \beta (h)]_\mathfrak{h},$
 \item[(II)] $\beta (\chi (x, y)) - \chi (\alpha(x), \alpha(y)) = \psi_{\alpha (x)} \lambda(y) - \psi_{\alpha (y)} \lambda(x)  - \lambda ([x, y]_\mathfrak{g}) + [\lambda (x) , \lambda(y)]_\mathfrak{h},$
 \item[(III)] $\beta (\Phi (x)) - \Phi (\alpha(x)) = S (\lambda(x)) - \lambda (T(x)),$ for $x, y \in \mathfrak{g}$ and $h \in \mathfrak{h}$.
\end{itemize}
\end{theorem}

\begin{proof}
Suppose the pair $(\beta, \alpha) \in \mathrm{Aut}(\mathfrak{h}_S) \times \mathrm{Aut}(\mathfrak{g}_T)$ is inducible. Thus, there exists an automorphism $\gamma \in \mathrm{Aut}_{\mathfrak{h}} (\mathfrak{e}_U)$ such that $\gamma|_{\mathfrak{h}} = \beta$ and $p \gamma s = \alpha$. For any $x \in \mathfrak{g}$, we observe that
\begin{align*}
p \big(  (\gamma s - s \alpha)(x)  \big) = \alpha (x) - \alpha (x) = 0 ~~~~ (\text{as } ps = \mathrm{id}_\mathfrak{g}).
\end{align*} 
This shows that $(\gamma s - s \alpha)(x) \in \mathrm{ker }~ p = \mathrm{im }~ i \cong \mathfrak{h}$. We define a map $\lambda : \mathfrak{g} \rightarrow \mathfrak{h}$ by
%\begin{align*}
$\lambda(x) := (\gamma s - s \alpha)(x), \text{ for } x \in \mathfrak{g}.$
%\end{align*}
Then we have
\begin{align*}
\beta (\psi_x h)- \psi_{\alpha (x)} \beta (h) =~& \beta ( [s(x), h]_\mathfrak{e}) - [s \alpha (x), \beta (h)]_\mathfrak{e} \\
=~& [\gamma s (x), \gamma (h)]_\mathfrak{e} - [s \alpha (x), \beta (h)]_\mathfrak{e} \quad (\text{as } \beta = \gamma|_\mathfrak{h})\\
=~& [\gamma s (x), \beta (h)]_\mathfrak{e} - [s \alpha (x), \beta (h)]_\mathfrak{e} \\
=~& [\lambda (x), \beta (h)]_\mathfrak{h}.
\end{align*}
Hence we get the identity (I). Next, we observe that
\begin{align*}
&\psi_{\alpha (x)} \lambda(y) - \psi_{\alpha (y)} \lambda(x) - \lambda([x, y]_\mathfrak{g}) + [\lambda (x), \lambda (y)]_\mathfrak{h} \\
&= [s \alpha (x) , (\gamma s - s \alpha)(y)]_\mathfrak{e} - [s \alpha (y), (\gamma s - s \alpha)(x)]_\mathfrak{e} - (\gamma s - s \alpha)([x, y]_\mathfrak{g}) + [(\gamma s - s \alpha)(x), (\gamma s - s \alpha)(y)]_\mathfrak{h} \\
&= [s \alpha (x), \gamma s (y)]_\mathfrak{e} - [s \alpha (x) , s \alpha (y)]_\mathfrak{e} - [s \alpha (y), \gamma s (x)]_\mathfrak{e} + [s \alpha (y), s \alpha(x)]_\mathfrak{e} - \gamma s ([x, y]_\mathfrak{g}) + s [\alpha (x), \alpha(y)]_\mathfrak{g} \\
& ~~ + [\gamma s (x), \gamma s (y)]_\mathfrak{e} - [\gamma s (x), s \alpha (y)]_\mathfrak{e} - [s \alpha (x), \gamma s (y)]_\mathfrak{e} + [s \alpha (x), s \alpha(y)]_\mathfrak{e} \\
&= \big( [\gamma s (x), \gamma s (y)]_\mathfrak{e} - \gamma s ([x, y]_\mathfrak{g}) \big) + [s \alpha (y), s \alpha (x)]_\mathfrak{e} + s [\alpha(x), \alpha (y)]_\mathfrak{g} \quad (\text{after cancellation and rearranging}) \\
&= \gamma \big(  [s(x), s(y)]_\mathfrak{e} - s [x, y]_\mathfrak{g} \big)  - \big( [s\alpha (x) , s \alpha(y)]_\mathfrak{e} - s [\alpha (x), \alpha(y)]_\mathfrak{g}  \big) \\
&= \beta \big(  [s(x), s(y)]_\mathfrak{e} - s [x, y]_\mathfrak{g} \big) - \big( [s\alpha (x) , s \alpha(y)]_\mathfrak{e} - s [\alpha (x), \alpha(y)]_\mathfrak{g}  \big) \\
&= \beta (\chi (x, y))- \chi (\alpha (x), \alpha (y)). 
\end{align*}
This proves the identity (II). Finally, the identity (III) follows as
\begin{align*}
&\beta (\Phi (x)) - \Phi (\alpha (x)) \\
&= \beta \big( U(s(x)) - s (T(x)) \big) - \big( U (s \alpha (x)) - s (T \alpha (x)) \big) \\
&= \gamma \big( U(s(x)) - s (T(x)) \big) -  U (s \alpha (x)) + s (T \alpha (x)) \quad (\text{as } \beta = \gamma|_\mathfrak{h}) \\
&= U \gamma (s(x)) - \gamma s (T(x)) - U (s \alpha (x)) + s ( \alpha T (x))  \quad (\text{as } U \gamma = \gamma U \text{ and } T \alpha = \alpha T) \\
&= U ((\gamma s - s \alpha)(x)) - (\gamma s - s \alpha)(T(x)) \\
&=U(\lambda(x)) - \lambda (T(x)) = S (\lambda(x)) - \lambda (T(x)). 
\end{align*}

Conversely, suppose there exists a linear map $\lambda \in \mathrm{Hom}(\mathfrak{g}, \mathfrak{h})$ satisfying (I), (II) and (III). Since $s$ is a section of the map $p$, it follows that any element $e \in \mathfrak{e}$ can be written as $e = h + s(x)$, for some $h \in \mathfrak{h}$ and $x \in \mathfrak{g}$. We define a map $\gamma : \mathfrak{e} \rightarrow \mathfrak{e}$ by
\begin{align}\label{g-defn}
\gamma (e) = \gamma (h + s(x)) := \big(  \beta (h) + \lambda (x) \big) + s (\alpha (x)).
\end{align}
We claim that $\gamma$ is bijective. If $\gamma (h + s(x)) = 0$ then it follows from (\ref{g-defn}) that $s (\alpha (x)) = 0$. Since $s$ and $\alpha$ are both injective, we have $x = 0$. Using this is (\ref{g-defn}), we obtain $\beta (h) = 0$ which implies $h = 0$. Therefore, we get $h + s(x) = 0$. This proves that $\gamma$ is injective. Finally, for any element $e = h + s(x) \in \mathfrak{e}$, we consider the element $e' = \big( \beta^{-1}(h) - \beta^{-1}\lambda \alpha^{-1} (x)  \big) + s(\alpha^{-1}(x)) \in \mathfrak{e}.$ Then we have 
\begin{align*}
\gamma (e') =  \big( \beta \beta^{-1} (h) - \beta \beta^{-1} \lambda \alpha^{-1} (x) + \lambda \alpha^{-1} (x) \big) + s(\alpha \alpha^{-1}(x))  = h + s(x) = e.
\end{align*}
This proves that $\gamma$ is surjective. Hence, we proved our claim. Next, for any two elements $e_1 = h_1 + s(x_1)$ and $e_2 = h_2 + s(x_2)$ of the vector space $\mathfrak{e}$, we have
\begin{align*}
&[\gamma (e_1), \gamma (e_2)]_\mathfrak{e} \\
&= [\gamma (h_1 + s(x_1)), \gamma (h_2 + s(x_2)) ]_\mathfrak{e} \\
&= [ \beta (h_1) + \gamma (x_1) + s (\alpha (x_1)) ~,~ \beta (h_2) + \gamma (x_2) + s (\alpha (x_2))  ]_\mathfrak{e}  \\
&= [ \beta (h_1), \beta (h_2)  ]_\mathfrak{e} + \underbrace{[\beta (h_1) , \lambda(x_2) ]_\mathfrak{e}}_A + \underbrace{[ \beta (h_1) , s (\alpha (x_2)) ]_\mathfrak{e}}_B + \underbrace{[ \lambda (x_1), \beta(h_2) ]_\mathfrak{e}}_C + \underbrace{[ \lambda (x_1), \lambda(x_2)]_\mathfrak{e}}_D \\
& ~~~ + \underbrace{[ \lambda (x_1), s (\alpha (x_2))  ]_\mathfrak{e}}_E + \underbrace{[ s(\alpha (x_1)), \beta (h_2)  ]_\mathfrak{e}}_F + \underbrace{[  s (\alpha (x_1)), \lambda (x_2)  ]_\mathfrak{e}}_G + \underbrace{[ s (\alpha (x_1)) ,   s (\alpha (x_2))]_\mathfrak{e}}_H 
\end{align*}
\begin{align*}
&= [ \beta (h_1), \beta (h_2)  ]_\mathfrak{e} \underbrace{- \beta (\psi_{x_2} h_1) + \psi_{\alpha (x_2)} \beta (h_1)}_A  \underbrace{ - \psi_{\alpha (x_2)} \beta (h_1)}_B + \underbrace{\beta (\psi_{x_1} h_2) - \psi_{\alpha (x_1)} \beta (h_2)}_C \\
& ~~~ + \underbrace{ \beta (\chi (x_1, x_2)) - \chi (\alpha (x_1), \alpha (x_2)) - \psi_{\alpha (x_1) \lambda (x_2)} + \psi_{\alpha (x_2)} \lambda (x_1) + \lambda ([x_1, x_2]_\mathfrak{g}) }_D \underbrace{- \psi_{\alpha (x_2)} \lambda (x_1)}_E \\
& ~~~ + \underbrace{\psi_{\alpha (x_1)} \beta (h_2)}_F + \underbrace{\psi_{\alpha (x_1)} \lambda (x_2)}_G + \underbrace{ \chi (\alpha (x_1), \alpha (x_2)) + s [\alpha (x_1), \alpha (x_2)]_\mathfrak{g}}_H \\
&= \beta \big(  [h_1, h_2]_\mathfrak{h} - \psi_{x_2} h_1 + \psi_{x_1} h_2 + \chi (x_1, x_2) \big) + \lambda ([x_1, x_2]_\mathfrak{g}) + s [\alpha(x_1), \alpha (x_2)]_\mathfrak{g} \quad (\text{after cancellations})\\
&= \beta \big(  [h_1, h_2]_\mathfrak{h} + [h_1 , s(x_2)]_\mathfrak{e} + [s(x_1), h_2]_\mathfrak{e} + \chi (x_1, x_2) \big) + \lambda ([x_1, x_2]_\mathfrak{g}) + s(\alpha [x_1, x_2]_\mathfrak{g}) \\
&= \gamma \big(  [h_1, h_2]_\mathfrak{h} + [h_1, s(x_2)]_\mathfrak{e} + [s (x_1), h_2]_\mathfrak{e} + \chi (x_1, x_2)  + s [x_1, x_2]_\mathfrak{g} \big) \quad (\text{definition of } \gamma)\\
&= \gamma \big(  [h_1, h_2]_\mathfrak{h} + [h_1, s(x_2)]_\mathfrak{e} + [s (x_1), h_2]_\mathfrak{e} + [s(x_1), s(x_2)]_\mathfrak{e}  \big) \\
&= \gamma \big(  [h_1 + s(x_1), h_2 + s(x_2)]_\mathfrak{e} \big) = \gamma ([e_1, e_2]_\mathfrak{e}).
\end{align*}
Therefore, $\gamma : \mathfrak{e} \rightarrow \mathfrak{e}$ is a Lie algebra homomorphism. Moreover, 
\begin{align*}
(\gamma \circ U)(e) &= (\gamma \circ U) (h + s(x)) \\
&= \gamma \big(  S(h) + Us (x)  \big) \\
&= \gamma \big( S(h) + \Phi (x) + s T(x)  \big) \quad (\text{since } \Phi = Us -sT)\\
&= \beta (S(h) + \Phi (x)) + \lambda (T(x)) + s (\alpha T(x)) \quad (\text{from the definition of } \gamma) \\
&= \beta S (h) + \beta \Phi (x) + \lambda (T(x)) + s (T \alpha (x)) \\
&= \beta S (h) + \beta \Phi (x) + \lambda (T(x)) + U (s \alpha (x)) - \Phi (\alpha (x)) \quad (\text{since } \Phi = Us -sT)\\
&= S \beta (h) + \big( \beta \Phi (x) + \lambda (T(x)) - \Phi(\alpha (x))   \big) + U (s \alpha (x)) \\
&= S \beta (h) + S (\lambda (x)) + U (s \alpha (x)) \qquad \qquad (\text{from  (III)})\\
&= U \big(    \beta (h) + \lambda (x) + s (\alpha (x)) \big) \quad (\text{as } U|_\mathfrak{h} = S) \\
&= U \circ \gamma (h + s(x)) = (U \circ \gamma)(e).
\end{align*}
This proves that $\gamma : \mathfrak{e}_U \rightarrow \mathfrak{e}_U$ is an automorphism of the Rota-Baxter Lie algebra $\mathfrak{e}_U$. In other words, $\gamma \in \mathrm{Aut}(\mathfrak{e}_U)$. Since $\gamma|_\mathfrak{h} \subset \mathfrak{h}$, we have $\gamma \in \mathrm{Aut}_{\mathfrak{h}} (\mathfrak{e}_U)$. Finally, we have
\begin{align*}
\gamma (h) =~& \gamma (h + s(0)) = \beta (h), \text{ for } h \in \mathfrak{h}, ~ \text{ and }\\
(p \gamma s)(x) =~& p \gamma (0 + s(x)) = p (\lambda (x) + s(\alpha(x))) = ps (\alpha (x)) = \alpha (x), \text{ for } x \in \mathfrak{g}.
\end{align*}
Hence $\gamma|_\mathfrak{h} = \beta$ and $p \gamma s = \alpha$. This proves that the pair $(\beta, \alpha) \in \mathrm{Aut} (\mathfrak{h}_S) \times \mathrm{Aut} (\mathfrak{g}_T)$ is inducible.
\end{proof}

%........................................................

\medskip

\medskip

Let $0 \rightarrow \mathfrak{h}_S \xrightarrow{i} \mathfrak{e}_U \xrightarrow{p} \mathfrak{g}_T \rightarrow 0$ be a non-abelian extension of $\mathfrak{g}_T$ by $\mathfrak{h}_S$ and let $(\chi, \psi, \Phi)$ be the corresponding non-abelian $2$-cocycle. For any $(\beta , \alpha) \in \mathrm{Aut}(\mathfrak{h}_S) \times \mathrm{Aut} (\mathfrak{g}_T)$, we define a triple $(\chi_{(\beta, \alpha)}, \psi_{(\beta, \alpha)}, \Phi_{(\beta, \alpha)})$ of linear maps $ \chi_{(\beta, \alpha)} : \wedge^2 \mathfrak{g} \rightarrow \mathfrak{h}$, ~$\psi_{( \beta, \alpha)} : \mathfrak{g} \rightarrow \mathrm{Der}(\mathfrak{h})$ and $\Phi_{(\beta, \alpha)} : \mathfrak{g} \rightarrow \mathfrak{h}$ by
\begin{align}\label{aut-2-defi}
\chi_{(\beta, \alpha)} (x, y) := \beta \circ \chi (\alpha^{-1}(x), \alpha^{-1}(y)), ~~~ (\psi_{(\beta, \alpha)})_x h = \beta (\psi_{\alpha^{-1}(x)} \beta^{-1}(h)) ~~~ \text{ and } ~~~ \Phi_{(\beta, \alpha)} (x) = \beta \Phi (\alpha^{-1} (x)),
\end{align}
for $x, y \in \mathfrak{g}$ and $h \in \mathfrak{h}$. Then we have the following.

\begin{proposition}
The triple $(\chi_{(\beta, \alpha)}, \psi_{(\beta, \alpha)}, \Phi_{(\beta, \alpha)})$ is a non-abelian $2$-cocycle.
\end{proposition}

\begin{proof}
Since $(\chi, \psi, \Phi)$ is a non-abelian $2$-cocycle, the identities (\ref{cocycle-condition I}), (\ref{cocycle-condition II}), (\ref{cocycle-condition III}) and (\ref{cocycle-condition IV}) are hold. In these identities, if we replace $x, y, h$ by $\alpha^{-1}(x), \alpha^{-1} (y), \beta^{-1}(h)$ respectively, and use the definitions (\ref{aut-2-defi}), we obtain the corresponding identities for the triple $(\chi_{(\beta, \alpha)}, \psi_{(\beta, \alpha)}, \Phi_{(\beta, \alpha)})$. For example, it follows from (\ref{cocycle-condition I}) that
\begin{align*}
\beta \big( \psi_{\alpha^{-1} x} \psi_{\alpha^{-1} y} \beta^{-1} h ~-~ \psi_{\alpha^{-1} y} \psi_{\alpha^{-1} x} \beta^{-1} h ~-~ \psi_{[\alpha^{-1}(x), \alpha^{-1}(y)]_\mathfrak{g}} \beta^{-1} h \big)= \beta [\chi (\alpha^{-1}(x), \alpha^{-1}(y)), \beta^{-1}h]_\mathfrak{h}
\end{align*}
which can be written as
\begin{align*}
\beta \big(   \psi_{\alpha^{-1}x} \beta^{-1} (\psi_{(\beta, \alpha)})_y h ~-~ \psi_{\alpha^{-1}y} \beta^{-1} (\psi_{(\beta, \alpha)})_x h ~-~ \psi_{\alpha^{-1} [x,y]_\mathfrak{g}} \beta^{-1} h \big) = [\beta \chi (\alpha^{-1}(x), \alpha^{-1}(y)), h ]_\mathfrak{h}.
\end{align*}
This is same as 
\begin{align*}
(\psi_{(\beta, \alpha)})_x (\psi_{(\beta, \alpha)})_y h ~-~ (\psi_{(\beta, \alpha)})_y (\psi_{(\beta, \alpha)})_x h ~-~ (\psi_{(\beta, \alpha)})_{[x,y]_\mathfrak{g}} h = [\chi_{(\beta, \alpha)} (x,y), h]_\mathfrak{h}.
\end{align*}
This shows that the identity (\ref{cocycle-condition I}) follows for the triple $(\chi_{(\beta, \alpha)}, \psi_{(\beta, \alpha)}, \Phi_{(\beta, \alpha)})$. 

This completes the proof.
\end{proof}

With the above notations, Theorem \ref{thm-inducibility} can be rephrased as follows.

\begin{theorem}\label{thm-inducibility-2}
Let $0 \rightarrow \mathfrak{h}_S \xrightarrow{i} \mathfrak{e}_U \xrightarrow{p} \mathfrak{g}_T \rightarrow 0$ be a non-abelian extension of Rota-Baxter Lie algebras. A pair of automorphisms $(\beta, \alpha) \in \mathrm{Aut}(\mathfrak{h}_S) \times \mathrm{Aut}(\mathfrak{g}_T)$ is inducible if and only if the non-abelian $2$-cocycles $(\chi_{(\beta, \alpha)}, \psi_{(\beta, \alpha)}, \Phi_{(\beta, \alpha)})$ and $(\chi, \psi, \Phi)$ are equivalent, that is, they corresponds to the same element in $H^2_{nab}(\mathfrak{g}_T, \mathfrak{h}_S)$.
\end{theorem}

\begin{proof}
Let the pair $(\beta, \alpha) \in \mathrm{Aut}(\mathfrak{h}_S) \times \mathrm{Aut}(\mathfrak{g}_T)$ be inducible. Then by Theorem \ref{thm-inducibility}, there exists a linear map $\lambda \in \mathrm{Hom} (\mathfrak{g}, \mathfrak{h})$ satisfying (I), (II) and (III). In these identities, replacing $x, y, h$ by $\alpha^{-1} (x), \alpha^{-1} (y), \beta^{-1} (h)$ respectively, we obtain
\begin{align*}
(\psi_{(\beta, \alpha)})_x h - \psi_x h =~& [\lambda \alpha^{-1} (x), h]_\mathfrak{h}, \\
\chi_{(\beta, \alpha)} (x,y) - \chi (x,y) =~& \psi_x \lambda \alpha^{-1} (y) - \psi_y \lambda \alpha^{-1} (x) - \lambda \alpha^{-1} ([x,y]_\mathfrak{g}) + [\lambda \alpha^{-1} (x), \lambda \alpha^{-1}(y)]_\mathfrak{h},\\
\Phi_{(\beta, \alpha)} (x) - \Phi (x) =~& S (\lambda \alpha^{-1} (x)) - \lambda \alpha^{-1} (T(x)).
\end{align*}
This shows that $(\chi_{(\beta, \alpha)}, \psi_{(\beta, \alpha)}, \Phi_{(\beta, \alpha)})$ and $(\chi, \psi, \Phi)$ are equivalent and the equivalence is given by the map $\lambda \alpha^{-1} : \mathfrak{g} \rightarrow \mathfrak{h}$.

Conversely, suppose that the non-abelian $2$-cocycles $(\chi_{(\beta, \alpha)}, \psi_{(\beta, \alpha)}, \Phi_{(\beta, \alpha)})$ and $(\chi, \psi, \Phi)$ are equivalent and the equivalence is given by a map $\varphi : \mathfrak{g} \rightarrow \mathfrak{h}$. Then it can be easily checked that the map $\lambda := \varphi \alpha : \mathfrak{g} \rightarrow \mathfrak{h}$ satisfies the conditions (I), (II), (III) of Theorem \ref{thm-inducibility}. Hence the pair $(\beta, \alpha) \in \mathrm{Aut} (\mathfrak{h}_S) \times \mathrm{Aut} (\mathfrak{g}_T)$ is inducible.
\end{proof}

\section{Wells exact sequence for Rota-Baxter Lie algebras}\label{sec-5}

In this section, we first define the Wells map associated with a non-abelian extension of Rota-Baxter Lie algebras. There is a close relationship between the inducibility of a pair of Rota-Baxter automorphisms and the image of the Wells map. Then we observe that the Wells map fits into a short exact sequence. Finally, we consider two other relevant short exact sequences.

\medskip

Let $0 \rightarrow \mathfrak{h}_S \xrightarrow{i} \mathfrak{e}_U \xrightarrow{p} \mathfrak{g}_T \rightarrow 0$ be a non-abelian extension of the Rota-Baxter Lie algebra $\mathfrak{g}_T$ by $\mathfrak{h}_S$. Let $(\chi, \psi, \Phi)$ be the corresponding non-abelian $2$-cocycle induced by a fixed section $s$. We define a set map $\mathcal{W} : \mathrm{Aut}(\mathfrak{h}_S) \times \mathrm{Aut} (\mathfrak{g}_T) \rightarrow H^2_{nab} (\mathfrak{g}_T, \mathfrak{h}_S)$ by
\begin{align}\label{wells-m}
\mathcal{W}((\beta, \alpha)) = [(\chi_{(\beta, \alpha)}, \psi_{(\beta, \alpha)}, \Phi_{(\beta, \alpha)}) - (\chi, \psi, \Phi)]
\end{align}
the equivalence class of $(\chi_{(\beta, \alpha)}, \psi_{(\beta, \alpha)}, \Phi_{(\beta, \alpha)}) - (\chi, \psi, \Phi)$. This map may not be a group homomorphism. The map $\mathcal{W}$ is called the {\bf Wells map}. 

\medskip

\begin{proposition}\label{wells-s-ind}
The Wells map $\mathcal{W}$ does not depend on the choice of section.
\end{proposition}

\begin{proof}
Let $s'$ be any other section with the induced non-abelian $2$-cocycle $(\chi', \psi', \Phi')$. Then we have seen in Section \ref{sec-3} that the non-abelian $2$-cocycles $(\chi, \psi, \Phi)$ and $(\chi', \psi', \Phi')$ are equivalent, and equivalence is given by the map $\varphi := s - s'$. 

\medskip

On the other hand, for any $x \in \mathfrak{g}$ and $h \in \mathfrak{h}$,
\begin{align*}
(\psi_{(\beta, \alpha)})_x (h) - (\psi'_{(\beta, \alpha)})_x (h) =~& \beta (\psi_{\alpha^{-1} (x)} \beta^{-1}(h)) - \beta (\psi'_{\alpha^{-1} (x)} \beta^{-1}(h)) \\
=~& \beta [  \varphi \alpha^{-1} (x), \beta^{-1}(h)]_\mathfrak{h} = [\beta \varphi \alpha^{-1} (x), h]_\mathfrak{h}.
\end{align*}
Similarly, by straightforward calculations, we have
\begin{align*}
\chi_{(\beta, \alpha)} (x,y) - \chi'_{(\beta, \alpha)} (x,y) =~& (\psi'_{(\beta, \alpha)})_x (\beta \varphi \alpha^{-1} (y)) - (\psi'_{(\beta, \alpha)})_y (\beta \varphi \alpha^{-1} (x)) \\
~& \qquad \qquad \qquad - \beta \varphi \alpha^{-1} ([x,y]_\mathfrak{g}) + [\beta \varphi \alpha^{-1} (x), \beta \varphi \alpha^{-1} (y)]_\mathfrak{h},
\end{align*}
\begin{align*}
\Phi_{(\beta, \alpha)} (x) - \Phi'_{(\beta, \alpha)} (x) = S \big( \beta \varphi \alpha^{-1} (x)  \big) - \beta \varphi \alpha^{-1} (T(x)),
\end{align*}
for $x, y \in \mathfrak{g}$. This shows that the non-abelian $2$-cocycles $(\chi_{(\beta, \alpha)}, \psi_{(\beta, \alpha)}, \Phi_{(\beta, \alpha)})$ and  $(\chi'_{(\beta, \alpha)}, \psi'_{(\beta, \alpha)}, \Phi'_{(\beta, \alpha)})$ are equivalent, and equivalence is given by $\beta \varphi \alpha^{-1}$.

\medskip

Combining the results of the last two paragraphs, we have that the non-abelian $2$-cocycles 
\begin{align*}
(\chi_{(\beta, \alpha)}, \psi_{(\beta, \alpha)}, \Phi_{(\beta, \alpha)}) - (\chi, \psi, \Phi) \qquad \text{ and } \qquad (\chi'_{(\beta, \alpha)}, \psi'_{(\beta, \alpha)}, \Phi'_{(\beta, \alpha)}) - (\chi', \psi', \Phi')
\end{align*}
and equivalent, and equivalence is given by $\beta \varphi \alpha^{-1} - \varphi$. Hence they corresponds to the same element in $H^2_{nab} (\mathfrak{g}_T, \mathfrak{h}_S)$. This completes the proof.
\end{proof}

\medskip

\begin{remark}
(i) It follows from Theorem \ref{thm-inducibility-2} that a pair of Rota-Baxter automorphisms $(\beta, \alpha) \in \mathrm{Aut}(\mathfrak{h}_S) \times \mathrm{Aut}(\mathfrak{g}_T)$ is inducible if and only if $\mathcal{W}((\beta, \alpha))$ is trivial. In other words, $\mathcal{W}((\beta, \alpha))$ is an obstruction for inducibility of the pair $(\beta, \alpha)$.

\medskip

(ii) Let $0 \rightarrow \mathfrak{h}_S \xrightarrow{i} \mathfrak{e}_U \xrightarrow{p} \mathfrak{g}_T \rightarrow 0$ be a non-abelian extension of Rota-Baxter Lie algebras and let $[(\chi, \psi, \Phi)] \in H^2_{nab} (\mathfrak{g}_T, \mathfrak{h}_S)$ be the corresponding equivalence class. Note that the group $\mathrm{Aut}(\mathfrak{h}_S) \times \mathrm{Aut}(\mathfrak{g}_T)$ acts on the space $H^2_{nab} (\mathfrak{g}_T, \mathfrak{h}_S)$ by
\begin{align*}
(\beta, \alpha) \cdot [(\overline{\chi}, \overline{\psi}, \overline{\Phi})] = [(\overline{\chi}_{(\beta, \alpha)}, \overline{\psi}_{(\beta, \alpha)}, \overline{\Phi}_{(\beta, \alpha)})],
\end{align*}
for any $(\beta, \alpha) \in \mathrm{Aut}(\mathfrak{h}_S) \times \mathrm{Aut}(\mathfrak{g}_T)$ and $[(\overline{\chi}, \overline{\psi}, \overline{\Phi})] \in H^2_{nab} (\mathfrak{g}_T, \mathfrak{h}_S)$. With this notation, the Wells map $\mathcal{W} : \mathrm{Aut}(\mathfrak{h}_S) \times \mathrm{Aut}(\mathfrak{g}_T) \rightarrow H^2_{nab} (\mathfrak{g}_T, \mathfrak{h}_S)$ is given by
\begin{align*}
\mathcal{W} ((\beta, \alpha)) = (\beta, \alpha) \cdot [(\chi, \psi, \Phi)] - [(\chi, \psi, \Phi)],
\end{align*}
for $(\beta, \alpha) \in \mathrm{Aut}(\mathfrak{h}_S) \times \mathrm{Aut}(\mathfrak{g}_T).$ This shows that the Wells map can be seen as a principal crossed homomorphism in the group cohomology complex of $\mathrm{Aut}(\mathfrak{h}_S) \times \mathrm{Aut}(\mathfrak{g}_T)$ with values in $H^2_{nab} (\mathfrak{g}_T, \mathfrak{h}_S)$.
\end{remark}

In the next theorem, we show that the Wells map $\mathcal{W}$ fits into an exact sequence. This generalizes the classical Wells exact sequence known in the literature.

\begin{theorem}\label{wells-seq}
Let $0 \rightarrow \mathfrak{h}_S \xrightarrow{i} \mathfrak{e}_U \xrightarrow{p} \mathfrak{g}_T \rightarrow 0$ be a non-abelian extension of Rota-Baxter Lie algebras. Then there is an exact sequence
\begin{align*}
1 \rightarrow \mathrm{Aut}_\mathfrak{h}^{\mathfrak{h}, \mathfrak{g}} (\mathfrak{e}_U) \xrightarrow{\iota} \mathrm{Aut}_\mathfrak{h} (\mathfrak{e}_U) \xrightarrow{\tau} \mathrm{Aut}(\mathfrak{h}_S) \times \mathrm{Aut}(\mathfrak{g}_T) \xrightarrow{\mathcal{W}} H^2_{nab}(\mathfrak{g}_T, \mathfrak{h}_S).
\end{align*}
Here $\mathrm{Aut}_\mathfrak{h}^{\mathfrak{h}, \mathfrak{g}} (\mathfrak{e}_U) = \{ \gamma \in \mathrm{Aut} (\mathfrak{e}_U) ~|~ \tau(\gamma) = (\mathrm{id}_\mathfrak{h}, \mathrm{id}_\mathfrak{g}) \}$.
\end{theorem} 

\begin{proof}
Since the inclusion map $\iota : \mathrm{Aut}_\mathfrak{h}^{\mathfrak{h}, \mathfrak{g}} (\mathfrak{e}_U) \rightarrow \mathrm{Aut}_\mathfrak{h} (\mathfrak{e}_U)$ is an injection, the above sequence is exact at the first term.

Next, we take an element $\gamma \in \mathrm{ker} (\tau)$. In other words, $\gamma \in \mathrm{Aut}_\mathfrak{h}(\mathfrak{e}_U)$ with the property that $\gamma|_\mathfrak{h} = \mathrm{id}_\mathfrak{h}$ and $p \gamma s = \mathrm{id}_\mathfrak{g}$. This shows that $\gamma \in \mathrm{Aut}_\mathfrak{h}^{\mathfrak{h}, \mathfrak{g}} (\mathfrak{e}_U)$. On the other hand, if $\gamma \in \mathrm{Aut}_\mathfrak{h}^{\mathfrak{h}, \mathfrak{g}} (\mathfrak{e}_U)$, then $\gamma \in \mathrm{ker}(\tau)$. Thus, we have $\mathrm{ker}(\tau) = \mathrm{Aut}_\mathfrak{h}^{\mathfrak{h}, \mathfrak{g}} (\mathfrak{e}_U) \cong \mathrm{im}(\iota)$. This shows that the sequence is exact at the second term.

Finally, to show that the sequence is exact at the third term, take a pair $(\beta, \alpha) \in \mathrm{ker}(\mathcal{W})$. Since the non-abelian $2$-cocycles $(\chi_{(\beta, \alpha)}, \psi_{(\beta, \alpha)}, \Phi_{(\beta, \alpha)})$ and $(\chi, \psi, \Phi)$ are equivalent, by Theorem \ref{thm-inducibility-2}, the pair $(\beta, \alpha)$ is inducible. In other words, there exists an automorphism $\gamma \in \mathrm{Aut}_\mathfrak{h}(\mathfrak{e}_U)$ such that $\tau (\gamma) = (\beta, \alpha)$. This shows that $(\beta, \alpha) \in \mathrm{im}(\tau)$. Conversely, if a pair $(\beta, \alpha) \in \mathrm{im}(\tau)$, then by definition the pair $(\beta, \alpha)$ is inducible. Hence again by Theorem \ref{thm-inducibility-2}, the non-abelian $2$-cocycles  $(\chi_{(\beta, \alpha)}, \psi_{(\beta, \alpha)}, \Phi_{(\beta, \alpha)})$ and $(\chi, \psi, \Phi)$ are equivalent. Therefore, $\mathcal{W}((\beta, \alpha)) = 0$ which implies that $(\beta, \alpha) \in \mathrm{ker}(\mathcal{W})$. Thus, we obtain $\mathrm{ker}(\mathcal{W}) = \mathrm{im}(\tau)$. Hence the result follows.
\end{proof}

\medskip

\medskip

\noindent {\bf Some other exact sequences.} There are two other questions that are relevant to Question A. More precisely, one may ask the following questions.

\medskip

\noindent {\bf Question B.} When a Rota-Baxter automorphism $\alpha \in \mathrm{Aut}(\mathfrak{g}_T)$ can be lifted to an automorphism in $\mathrm{Aut}(\mathfrak{e}_U)$ fixing $\mathfrak{h}$ pointwise? Given a Rota-Baxter automorphism $\alpha \in \mathrm{Aut}(\mathfrak{g}_T)$, this question essentially asks: When the pair $(\mathrm{id}_\mathfrak{h}, \alpha) \in \mathrm{Aut}(\mathfrak{h}_S) \times \mathrm{Aut}(\mathfrak{g}_T)$ is inducible?

\medskip

\noindent {\bf Question C.} When a Rota-Baxter automorphism $\beta \in \mathrm{Aut}(\mathfrak{h}_S)$ can be extended to an automorphism in $\mathrm{Aut}(\mathfrak{e}_U)$ inducing the identity map on $\mathfrak{g}$? Given a Rota-Baxter automorphism $\beta \in \mathrm{Aut}(\mathfrak{h}_S)$, this question asks: When the pair $(\beta, \mathrm{id}_\mathfrak{g})$ is inducible?

\medskip

In the following, we will answer the above questions. To address these, we first set-up some notations. Let
\begin{align*}
\mathrm{Aut}_{\mathfrak{h}}^{\mathfrak{h}} :=~& \{ \gamma \in \mathrm{Aut}_\mathfrak{h} (\mathfrak{e}_U) ~|~ \gamma|_\mathfrak{h} = \mathrm{id}_\mathfrak{h} \},\\
\mathrm{Aut}_{\mathfrak{h}}^{\mathfrak{g}} :=~& \{ \gamma \in \mathrm{Aut}_\mathfrak{h} (\mathfrak{e}_U) ~|~ \overline{\gamma} :=  p \gamma s = \mathrm{id}_\mathfrak{g} \}
\end{align*} 
be two subgroups of $\mathrm{Aut}_\mathfrak{h} (\mathfrak{e}_U)$. Note that $\mathrm{Aut}_\mathfrak{h}^\mathfrak{h} (\mathfrak{e}_U)$ consist of all automorphisms in $\mathrm{Aut}_\mathfrak{h} (\mathfrak{e}_U)$ that induces identity map on $\mathfrak{h}$, and $\mathrm{Aut}_\mathfrak{h}^\mathfrak{g} (\mathfrak{e}_U)$ consist of all automorphisms in $\mathrm{Aut}_\mathfrak{h} (\mathfrak{e}_U)$ that induces identity map on $\mathfrak{g}$. 
There are obvious maps $\tau_1 : \mathrm{Aut}_\mathfrak{h}^\mathfrak{h} (\mathfrak{e}_U) \rightarrow \mathrm{Aut}(\mathfrak{g}_T)$ and $\tau_2 : \mathrm{Aut}_\mathfrak{h}^\mathfrak{g} (\mathfrak{e}_U) \rightarrow \mathrm{Aut}(\mathfrak{h}_S)$ given by
\begin{align*}
\tau_1 (\gamma) :=~& \overline{\gamma}, ~\text{ for } \gamma \in \mathrm{Aut}_\mathfrak{h}^\mathfrak{h} (\mathfrak{e}_U),\\
\tau_2 (\gamma) :=~& \gamma|_\mathfrak{h}, ~ \text{ for } \gamma \in \mathrm{Aut}_\mathfrak{h}^\mathfrak{g} (\mathfrak{e}_U).
\end{align*}

In similar to the Wells map (\ref{wells-m}), we define maps $\mathcal{W}_\mathfrak{g} : \mathrm{Aut}(\mathfrak{g}_T) \rightarrow H^2_{nab} (\mathfrak{g}_T, \mathfrak{h}_S)$ and $\mathcal{W}_\mathfrak{h} : \mathrm{Aut}(\mathfrak{h}_S) \rightarrow H^2_{nab} (\mathfrak{g}_T, \mathfrak{h}_S)$ by
\begin{align*}
\mathcal{W}_\mathfrak{g}  (\alpha) := [   ( \chi_{(\mathrm{id}_\mathfrak{h}, \alpha)}, \psi_{(\mathrm{id}_\mathfrak{h}, \alpha)}, \Phi_{(\mathrm{id}_\mathfrak{h}, \alpha) }) - (\chi, \psi, \Phi)],\\
\mathcal{W}_\mathfrak{h} (\beta) := [   ( \chi_{( \beta, \mathrm{id}_\mathfrak{g})}, \psi_{(\beta, \mathrm{id}_\mathfrak{g})}, \Phi_{(\beta, \mathrm{id}_\mathfrak{g}) }) - (\chi, \psi, \Phi)],
\end{align*}
for $\alpha \in \mathrm{Aut}(\mathfrak{g}_T)$ and $\beta \in \mathrm{Aut}(\mathfrak{h}_S)$. Similar to Proposition \ref{wells-s-ind}, one can check that the maps $\mathcal{W}_\mathfrak{g}$ and $\mathcal{W}_\mathfrak{h}$ doesn't depend on the choice of the section. The following observations show the importance of the maps $\mathcal{W}_\mathfrak{g}$ and $\mathcal{W}_\mathfrak{h}$.

\medskip

- A Rota-Baxter automorphism $\alpha \in \mathrm{Aut} (\mathfrak{g}_T)$ can be lifted to an automorphism in $\mathrm{Aut}(\mathfrak{e}_U)$ fixing $\mathfrak{h}$ pointwise if and only if $\mathcal{W}_\mathfrak{g} (\alpha) = 0$.

\medskip

- A Rota-Baxter automorphism $\beta \in \mathrm{Aut} (\mathfrak{h}_S)$ can be extended to an automorphism in $\mathrm{Aut} (\mathfrak{e}_U)$ inducing the identity on $\mathfrak{g}$ if and only if $\mathcal{W}_\mathfrak{h} (\beta) = 0$.

\medskip

With all the above notations, we have the following result.

\begin{theorem}
Let $0 \rightarrow \mathfrak{h}_S \xrightarrow{i} \mathfrak{e}_U \xrightarrow{p} \mathfrak{g}_T \rightarrow 0$ be a non-abelian extension of Rota-Baxter Lie algebras. Then there are two exact sequences
\begin{align}
1 \rightarrow \mathrm{Aut}_{\mathfrak{h}}^{\mathfrak{h}, \mathfrak{g}} (\mathfrak{e}_U) \xrightarrow{\iota} \mathrm{Aut}_\mathfrak{h}^\mathfrak{h} (\mathfrak{e}_U) \xrightarrow{\tau_1} \mathrm{Aut}(\mathfrak{g}_T) \xrightarrow{\mathcal{W}_\mathfrak{g}} H^2_{nab} (\mathfrak{g}_T, \mathfrak{h}_S) \label{exact-seq-a}\\
1 \rightarrow \mathrm{Aut}_{\mathfrak{h}}^{\mathfrak{h}, \mathfrak{g}} (\mathfrak{e}_U) \xrightarrow{\iota} \mathrm{Aut}_\mathfrak{h}^\mathfrak{g} (\mathfrak{e}_U) \xrightarrow{\tau_2} \mathrm{Aut}(\mathfrak{h}_S) \xrightarrow{\mathcal{W}_\mathfrak{h}} H^2_{nab} (\mathfrak{g}_T, \mathfrak{h}_S). \label{exact-seq-b}
\end{align}
\end{theorem}

\begin{proof}
Note that $\mathrm{Aut}_{\mathfrak{h}}^{\mathfrak{h}, \mathfrak{g}} (\mathfrak{e}_U)$ is a subgroup of both $\mathrm{Aut}_\mathfrak{h}^\mathfrak{h} (\mathfrak{e}_U)$ and $\mathrm{Aut}_\mathfrak{h}^\mathfrak{g} (\mathfrak{e}_U)$. The map $\iota$ (in the above sequences) is just the inclusion map.

\medskip

We first prove the exactness of the sequence (\ref{exact-seq-a}).  Let $\gamma \in \mathrm{ker }(\tau_1)$. Then we have $\gamma|_\mathfrak{h} = \mathrm{id}_\mathfrak{h}$ and $\overline{\gamma} = \mathrm{id}_\mathfrak{g}$ which means that $\gamma \in \mathrm{Aut}_\mathfrak{h}^{\mathfrak{h}, \mathfrak{g}} (\mathfrak{e}_U) \cong \mathrm{im }(i)$. On the other hand, if $\gamma \in \mathrm{Aut}_\mathfrak{h}^{\mathfrak{h}, \mathfrak{g}} (\mathfrak{e}_U)$, then by definition $\gamma|_\mathfrak{h} = \mathrm{id}_\mathfrak{h}$ and $\overline{\gamma} = \mathrm{id}_\mathfrak{g}$, which means that $\gamma \in \mathrm{ker}(\tau_1)$. Therefore, $\mathrm{ker} (\tau_1) = \mathrm{Aut}_\mathfrak{h}^{\mathfrak{h}, \mathfrak{g}} (\mathfrak{e}_U)$ which  shows the sequence is exact at the second term. Next, we take an element $\alpha \in \mathrm{ker } (\mathcal{W}_\mathfrak{g})$. This shows that the non-abelian $2$-cocycles $(\chi_{(\mathrm{id}_\mathfrak{h}, \alpha)}, \psi_{(\mathrm{id}_\mathfrak{h}, \alpha)}, \Phi_{(\mathrm{id}_\mathrm{h}, \alpha)})$ and $(\chi, \psi, \Phi)$ are equivalent. In other words, the pair $(\mathrm{id}_\mathfrak{h}, \alpha)$ is inducible. Therefore, there exists a map $\gamma \in \mathrm{Aut}_\mathfrak{h} (\mathfrak{e}_U)$ such that $\gamma|_\mathfrak{h} = \mathrm{id}_\mathfrak{h}$ and $\overline{\gamma} = \alpha$. This shows that $\gamma \in \mathrm{Aut}_\mathfrak{h}^\mathfrak{h} (\mathfrak{e}_U)$ with $\tau_1 (\gamma) = \alpha$. Hence $\alpha \in \mathrm{im}(\tau_1)$. Conversely, if $\alpha \in \mathrm{im}(\tau_1)$, then by definition the pair $(\mathrm{id}_\mathfrak{h}, \alpha)$ is inducible. Hence by Theorem \ref{thm-inducibility-2}, $( \chi_{(\mathrm{id}_\mathrm{h}, \alpha)}, \psi_{(\mathrm{id}_\mathrm{h}, \alpha)}, \Phi_{(\mathrm{id}_\mathfrak{h}, \alpha)}  )$ and $(\chi, \psi, \Phi)$ are equivalent which implies $\mathcal{W}_\mathfrak{g} (\alpha) = 0$. Thus $\alpha \in \mathrm{ker }(\mathcal{W}_\mathfrak{g})$. In summary, we obtain $\mathrm{ker}(\mathcal{W}_\mathfrak{g}) = \mathrm{im}(\tau_1)$. Hence the sequence is exact at the third term.

\medskip

The verification of the exactness of (\ref{exact-seq-b}) is similar. Hence we omit the details.
\end{proof}

\medskip

\section{Particular case: Abelian extensions of Rota-Baxter Lie algebras}\label{sec-6}

In this section, we show how the results of previous sections fit with the abelian extensions of Rota-Baxter Lie algebras. In particular, we recover the classification result of abelian extensions and obtain the corresponding Wells exact sequence.

\medskip

Let $\mathfrak{g}_T$ be a Rota-Baxter Lie algebra and $\mathfrak{h}_S$ be a vector space equipped with a linear map (not necessarily a representation). Consider $\mathfrak{h}_S$ as an abelian Rota-Baxter Lie algebra. An {\bf abelian extension} of $\mathfrak{g}_T$ by $\mathfrak{h}_S$ is a short exact sequence 
\begin{align}
\label{ab_ext}
\xymatrix{
0 \ar[r] & \mathfrak{h}_S \ar[r]^{i} & \mathfrak{e}_U \ar[r]^{p} & \mathfrak{g}_T \ar[r] & 0.
}
\end{align} 
of Rota-Baxter Lie algebras. Equivalences between abelian extensions of $\mathfrak{g}_T$ by $\mathfrak{h}_S$ can be defined similar to the Definition \ref{defn-nab-equiv}.

For any section $s : \mathfrak{g} \rightarrow \mathfrak{e}$ of the map $p$, there is a $\mathfrak{g}$-module structure on $\mathfrak{h}$ given by $\psi : \mathfrak{g} \rightarrow \mathrm{End}(\mathfrak{h})$, $\psi_x h = [ s(x), h]_\mathfrak{e}$, for $x \in \mathfrak{g}, h \in \mathfrak{h}$. It is easy to verify that this $\mathfrak{g}$-module structure on $\mathfrak{h}$ does not depend on the choice of section. Moreover, with respect to the above $\mathfrak{g}$-module structure, $\mathfrak{h}_S$ becomes a representation of the Rota-Baxter Lie algebra $\mathfrak{g}_T$ \cite{jiang-sheng}, called the induced representation.

\medskip

Given any representation $\mathfrak{h}_S$ of the Rota-Baxter Lie algebra $\mathfrak{g}_T$, let $\mathrm{Ext}_{ab}(\mathfrak{g}_T, \mathfrak{h}_S)$ be the set of all equivalence classes of abelian extensions of $\mathfrak{g}_T$ by $\mathfrak{h}_S$ so that the induced representation on $\mathfrak{h}_S$ coincides with the given one. The following classification result about abelian extensions has been proved in \cite{jiang-sheng}. We show that it can be deduced from our classification result of non-abelian extensions (see Section \ref{sec-3}).

\begin{theorem}
Let $\mathfrak{g}_T$ be a Rota-Baxter Lie algebra and $\mathfrak{h}_S$ be a representation of it. Then 
\begin{align*}
\mathrm{Ext}_{ab} (\mathfrak{g}_T, \mathfrak{h}_S ) ~\cong ~ H^2_{\mathrm{RBLie}} (\mathfrak{g}_T, \mathfrak{h}_S).
\end{align*}
\end{theorem}

\begin{proof}
Let $\mathfrak{e}_U$ be an abelian extension of $\mathfrak{g}_T$ by $\mathfrak{h}_S$ representing an element in $\mathrm{Ext}_{ab} (\mathfrak{g}_T, \mathfrak{h}_S)$. Since $[~,~]_\mathfrak{h} = 0$, it follows from (\ref{cocycle-condition I}) that $\psi$ defines a $\mathfrak{g}$-module structure on $\mathfrak{h}$ which is the prescribed one, and the condition (\ref{cocycle-condition III}) implies that $\mathfrak{h}_S$ is a representation of the Rota-Baxter Lie algebra $\mathfrak{g}_T$. The condition (\ref{cocycle-condition II}) says that $\chi : \wedge^2 \mathfrak{g} \rightarrow \mathfrak{h}$ is a $2$-cocycle on the Chevalley-Eilenberg cochain complex of $\mathfrak{g}$ with coefficients in $\mathfrak{h}$ (i.e. $\delta^2 (\chi) = 0$). Finally, the condition (\ref{cocycle-condition IV}) translates into
\begin{align*}
\partial^1 (\Phi) (x, y) + \chi (Tx, Ty)- S \big(  \chi (Tx, y) + \chi (x, Ty) \big) = 0.
\end{align*}
Thus, we have
\begin{align*}
\delta^2_{\text{RBLie}} (\chi, \Phi) = \big( \delta^2 (\chi),~ \partial^1 (\Phi) + \chi (T \otimes T) - S (\chi (T \otimes \mathrm{id}) + \chi (\mathrm{id} \otimes T))  \big) = 0.
\end{align*}
Hence $(\chi, \Phi)$ is a $2$-cocycle in the cohomology complex of the Rota-Baxter Lie algebra $\mathfrak{g}_T$ with coefficients in $\mathfrak{h}_S$ (see Section \ref{sec-2}). Moreover, if $\mathfrak{e}_U$ and $\mathfrak{e}'_{U'}$ are two equivalent abelian extensions representing same element in $\mathrm{Ext}_{ab} (\mathfrak{g}_T, \mathfrak{h}_S)$, then the corresponding $2$-cocycles $(\chi, \Phi)$ and $(\chi', \Phi')$ are cohomologous. In other words, one obtains a map $\Upsilon: \mathrm{Ext}_{ab} (\mathfrak{g}_T, \mathfrak{h}_S) \rightarrow H^2_{\mathrm{RBLie}} (\mathfrak{g}_T, \mathfrak{h}_S)$.

\medskip

Conversely, following Proposition \ref{prop-nab-ext} with $[~,~]_\mathfrak{h} = 0$, one can show that any cohomology class in $H^2_{\mathrm{RBLie}} (\mathfrak{g}_T, \mathfrak{h}_S)$ gives rise to an equivalence class of abelian extensions of $\mathfrak{g}_T$ by $\mathfrak{h}_S$. Hence there is a map $\Omega : H^2_{\mathrm{RBLie}} (\mathfrak{g}_T, \mathfrak{h}_S) \rightarrow \mathrm{Ext}_{ab} (\mathfrak{g}_T, \mathfrak{h}_S)$. Finally, the maps $\Upsilon$ and $\Omega$ are inverses to each other. This completes the proof.
\end{proof}

\medskip

%********************************************************

Let $\mathfrak{g}_T$ be a Rota-Baxter Lie algebra and $\mathfrak{h}_S$ be a representation of it (with the $\mathfrak{g}$-module structure on $\mathfrak{h}$ given by $\psi : \mathfrak{g} \rightarrow \mathrm{End}(\mathfrak{h})$). Let $0 \rightarrow \mathfrak{h}_S \xrightarrow{i} \mathfrak{e}_U \xrightarrow{p} \mathfrak{g}_T \rightarrow 0$ be an abelian extension representing an element in $\mathrm{Ext}_{ab} (\mathfrak{g}_T, \mathfrak{h}_S)$. For any fixed section $s : \mathfrak{g} \rightarrow \mathfrak{e}$, by definition the $\mathfrak{g}$-module structure $\psi$ is given by $\psi_x h = [s(x), h]_\mathfrak{e}$, for $x \in \mathfrak{g}$, $h \in \mathfrak{h}$. Let $(\chi, \Phi)$ be the $2$-cocycle associated to the above abelian extension depending on the section $s$ (see the above Theorem). With all these set-up, Theorem \ref{thm-inducibility} translates into the following.

\begin{theorem}
Let $0 \rightarrow \mathfrak{h}_S \xrightarrow{i} \mathfrak{e}_U \xrightarrow{p} \mathfrak{g}_T \rightarrow 0$ be an abelian extension of the Rota-Baxter Lie algebra $\mathfrak{g}_T$ by a representation $\mathfrak{h}_S$. A pair $(\beta, \alpha) \in \mathrm{Aut}(\mathfrak{h}_S) \times \mathrm{Aut}(\mathfrak{g}_T)$ is inducible if and only if
\begin{align*}
\beta (\psi_x h) - \psi_{\alpha(x)} \beta (h) = 0
\end{align*}
and there exists a linear map $\lambda \in \mathrm{Hom}(\mathfrak{g}, \mathfrak{h})$ satisfying
\begin{align*}
\beta (\chi (x,y)) - \chi (\alpha(x), \alpha(y)) =~& \psi_{\alpha (x)} \lambda(y) - \psi_{\alpha(y)} \lambda (x) -\lambda ([x,y]_\mathfrak{g}),\\
\beta (\Phi (x)) -\Phi (\alpha (x)) =~& S (\lambda(x)) - \lambda (T(x)),~\text{ for } x, y \in \mathfrak{g} \text{ and } h \in \mathfrak{h}.
\end{align*}
\end{theorem}

\medskip

Let $0 \rightarrow \mathfrak{h}_S \xrightarrow{i} \mathfrak{e}_U \xrightarrow{p} \mathfrak{g}_T \rightarrow 0$ be an abelian extension of the Rota-Baxter Lie algebra $\mathfrak{g}_T$ by a representation $\mathfrak{h}_S$, and let $\psi$ denotes the corresponding $\mathfrak{g}$-module structure on $\mathfrak{h}$. Define
\begin{align*}
C_\psi = \{ (\beta, \alpha) \in \mathrm{Aut}(\mathfrak{h}_S) \times \mathrm{Aut}(\mathfrak{g}_T) ~|~ \beta (\psi_x h) = \psi_{\alpha (x)} \beta (h), \text{ for all } x \in \mathfrak{g}, h \in \mathfrak{h} \}.
\end{align*}
The space $C_\psi$ is called the space of {\bf compatible pairs} of automorphisms. It is easy to see that $C_\psi$ is a subgroup of $\mathrm{Aut}(\mathfrak{h}_S) \times \mathrm{Aut}(\mathfrak{g}_T)$.

Let $(\beta, \alpha) \in \mathrm{Aut}(\mathfrak{h}_S) \times \mathrm{Aut}(\mathfrak{g}_T)$. Then the map $\alpha$ induces a new $\mathfrak{g}$-module structure $\psi^\alpha : \mathfrak{g} \rightarrow \mathrm{End}(\mathfrak{h})$ on $\mathfrak{h}$ given by $(\psi^\alpha)_x h := \psi_{\alpha (x)} h$, for $x \in \mathfrak{g}$, $h \in \mathfrak{h}$. This infact gives rise to a representation of the Rota-Baxter Lie algebra $\mathfrak{g}_T$ on $\mathfrak{h}_S$ as
\begin{align*}
(\psi^\alpha)_{T(x)} S(h) = \psi_{\alpha T(x)} S(h) =~& \psi_{T \alpha (x)} S(h) \\
=~& S (\psi_{T\alpha(x)} h + \psi_{\alpha (x)} S(h)) \\
=~& S \big(  (\psi^\alpha)_{T(x)} h + (\psi^\alpha)_x S(h) \big), \text{ for } x \in \mathfrak{g}, h \in \mathfrak{h}.
\end{align*} 
We denote this representation by $\mathfrak{h}^\alpha_S$. Then we have the following.

\begin{lemma}
A pair $(\beta, \alpha) \in \mathrm{Aut}(\mathfrak{h}_S) \times \mathrm{Aut}(\mathfrak{g}_T)$ is in $C_\psi$ if and only if the linear map $\beta : \mathfrak{h} \rightarrow \mathfrak{h}$ is a morphism of Rota-Baxter representations from $\mathfrak{h}_S$ to $\mathfrak{h}^\alpha_S$.
\end{lemma}

Let $0 \rightarrow \mathfrak{h}_S \xrightarrow{i} \mathfrak{e}_U \xrightarrow{p} \mathfrak{g}_T \rightarrow 0$ be an abelian extension of the Rota-Baxter Lie algebra $\mathfrak{g}_T$ by a representation $\mathfrak{h}_S$ (with the $\mathfrak{g}$-module structure on $\mathfrak{h}$ given by $\psi : \mathfrak{g} \rightarrow \mathrm{End}(\mathfrak{h})$). Given a section $s$, let $(\chi, \Phi)$ be the $2$-cocycle associated to abelian extension. Let $(\beta, \alpha) \in \mathrm{Aut}(\mathfrak{h}_S) \times \mathrm{Aut}(\mathfrak{g}_T)$ be a pair of Rota-Baxter automorphisms. Since in the case of abelian extensions, we are not changing the $\mathfrak{g}$-module structure $\psi$ (unlike non-abelian case where we allow to change $\psi$ to $\psi_{(\beta, \alpha)}$), the pair $(\chi_{(\beta, \alpha)}, \Phi_{(\beta, \alpha)})$ may not be a $2$-cocycle. However, if $(\beta, \alpha) \in C_\psi$, then we observe that
\begin{align*}
\delta^2 (\chi_{(\beta, \alpha)}) (x,y,z) =~& \psi_x (\chi_{(\beta, \alpha)} (y,z)) + ~c.p.~- \chi_{(\beta, \alpha)} ([x,y]_\mathfrak{g}, z) -~ c.p. \\
=~& \psi_x \beta \chi (\alpha^{-1}y, \alpha^{-1}z) + ~c.p.~ - \beta \chi (\alpha^{-1} [x,y]_\mathfrak{g}, \alpha^{-1}z) -~c.p.\\
=~& \beta \psi_{\alpha^{-1}x} \chi (\alpha^{-1}y, \alpha^{-1}z) + ~c.p.~- \beta \chi ([\alpha^{-1}x, \alpha^{-1}y]_\mathfrak{g}, \alpha^{-1}z) -~c.p.   \qquad (\text{as } (\beta, \alpha) \in C_\psi) \\
=~& \beta (\delta^2 (\chi)) (\alpha^{-1}x, \alpha^{-1} y, \alpha^{-1} z) = 0 \qquad (\text{as } \delta^2 (\chi) = 0)
\end{align*}
and
\begin{align*}
&\partial^1 (\Phi_{(\beta, \alpha)}) (x,y) + \chi_{(\beta, \alpha)} (Tx, Ty) - S \chi_{(\beta, \alpha)} (Tx, y) - S \chi_{(\beta, \alpha)} (x, Ty) \\
&= \psi_{T(x)} \Phi_{(\beta, \alpha)} (y) - S (\psi_x \Phi_{(\beta, \alpha)} (y)) - \Phi_{(\beta, \alpha)} ([Tx,y]_\mathfrak{g} + [x, Ty]_\mathfrak{g}) - \psi_{T(y)} \Phi_{(\beta, \alpha)} (x) + S (\psi_y \Phi_{(\beta, \alpha)} (x)) \\
& \quad + \beta \chi (\alpha^{-1} Tx, \alpha^{-1} Ty) - S \beta \chi (\alpha^{-1} Tx, \alpha^{-1}y) - S \beta \chi (\alpha^{-1}x, \alpha^{-1} Ty) \\
&= \psi_{T(x)} \beta \Phi \alpha^{-1}(y) - S (\psi_x \beta \Phi \alpha^{-1}(y)) - \beta \Phi \alpha^{-1} ( [Tx,y]_\mathfrak{g} + [x, Ty]_\mathfrak{g} ) - \psi_{T(y)} \beta \Phi \alpha^{-1}(x) + S (\psi_y \beta \Phi \alpha^{-1}(x)) \\
& \quad + \beta \chi (T \alpha^{-1}x, T \alpha^{-1} y) - \beta S \chi (T \alpha^{-1}x, \alpha^{-1}y) - \beta S \chi (\alpha^{-1}x, T \alpha^{-1} y) \\
&= \beta \bigg(  \psi_{T \alpha^{-1} (x)} \Phi \alpha^{-1} (y) - S (\psi_{\alpha^{-1}(x)} \Phi \alpha^{-1}(y)) - \Phi ([T \alpha^{-1}x, \alpha^{-1}y]_\mathfrak{g} + [ \alpha^{-1}x,T \alpha^{-1}y]_\mathfrak{g} ) \\
& \quad - \psi_{T \alpha^{-1} (y)} \Phi \alpha^{-1} (x) + S (\psi_{\alpha^{-1}(y)} \Phi \alpha^{-1}(x)) + \chi (T \alpha^{-1}x, T \alpha^{-1}y) - S \chi (T\alpha^{-1}x, \alpha^{-1}y) - S \chi (\alpha^{-1}x, T \alpha^{-1}y) \bigg) \\
&= \beta \big(  \underbrace{\partial^1 (\Phi) + \chi \circ (T \otimes T) - S \chi (T \otimes \mathrm{id}) - S \chi (\mathrm{id} \otimes T) }_{= 0} \big) (\alpha^{-1}x,\alpha^{-1}y) = 0.
\end{align*}
Therefore, 
\begin{align*}
\delta^2_\mathrm{RBLie} (\chi_{(\beta, \alpha)} , \Phi_{(\beta, \alpha)}) = \big( \delta^2 (\chi_{(\beta, \alpha)}),~  \partial^1 (\Phi_{(\beta, \alpha)}) + \chi_{(\beta, \alpha)} \circ (T \otimes T) - S \chi_{(\beta, \alpha)} (T \otimes \mathrm{id}) - S \chi_{(\beta, \alpha)} (\mathrm{id} \otimes T) \big) = 0
\end{align*}
which shows that $(\chi_{(\beta, \alpha)}, \Phi_{(\beta, \alpha)})$ is a $2$-cocycle. With the above observations, Theorem \ref{thm-inducibility-2} has the following form in the abelian context.

\begin{theorem}\label{abl-ext1}
Let $0 \rightarrow \mathfrak{h}_S \xrightarrow{i} \mathfrak{e}_U \xrightarrow{p} \mathfrak{g}_T \rightarrow 0$ be an abelian extension of the Rota-Baxter Lie algebra $\mathfrak{g}_T$ by a representation $\mathfrak{h}_S$. A pair $(\beta, \alpha) \in \mathrm{Aut}(\mathfrak{h}_S) \times \mathrm{Aut}(\mathfrak{g}_T)$ is inducible if and only if
\begin{enumerate}
\item $(\beta, \alpha) \in C_\psi$,
\item the $2$-cocycles $(\chi_{(\beta, \alpha)}, \Phi_{(\beta, \alpha)})$ and $(\chi, \Phi)$ are cohomologous.
\end{enumerate}
\end{theorem}

\medskip

For any $(\beta, \alpha) \in \mathrm{Aut} (\mathfrak{h}_S) \times \mathrm{Aut} (\mathfrak{g}_T)$, the pair $(\chi_{(\beta, \alpha)}, \Phi_{(\beta, \alpha)})$ need not to be a $2$-cocycle. However, it has been observed that $(\chi_{(\beta, \alpha)}, \Phi_{(\beta, \alpha)})$ is a $2$-cocycle if $(\beta, \alpha) \in C_\psi$. This allows us to consider the {\bf Wells map} in the abelian context as a set map $\mathcal{W} : C_\psi \rightarrow H^2_\mathrm{RBLie} (\mathfrak{g}_T, \mathfrak{h}_S)$ defined by
\begin{align}\label{w-abel}
\mathcal{W} ((\beta, \alpha)) = [   (\chi_{(\beta, \alpha)}, \Phi_{(\beta, \alpha)}) - (\chi, \Phi)], \text{ for } (\beta, \alpha) \in C_\psi.
\end{align}
It follows from Theorem \ref{abl-ext1} that a pair $(\beta, \alpha) \in \mathrm{Aut} (\mathfrak{h}_S) \times \mathrm{Aut}(\mathfrak{g}_T)$ is inducible if and only if $(\beta, \alpha) \in C_\psi$ and $\mathcal{W} ((\beta, \alpha)) = 0.$

Since the image of the map $\tau : \mathrm{Aut}_\mathfrak{h} (\mathfrak{e}_U) \rightarrow \mathrm{Aut}(\mathfrak{h}_S) \times \mathrm{Aut}(\mathfrak{g}_T)$ lies in the subgroup $C_\psi \subset \mathrm{Aut}(\mathfrak{h}_S) \times \mathrm{Aut}(\mathfrak{g}_T)$, we get the following Wells short exact sequence (deduce from Theorem \ref{wells-seq}).

\begin{theorem}\label{wells-seq-abl}
Let $0 \rightarrow \mathfrak{h}_S \xrightarrow{i} \mathfrak{e}_U \xrightarrow{p} \mathfrak{g}_T \rightarrow 0$ be an abelian extension of the Rota-Baxter Lie algebra $\mathfrak{g}_T$ by a representation $\mathfrak{h}_S$. Then there is an exact sequence
\begin{align*}
1 \rightarrow \mathrm{Aut}_\mathfrak{h}^{\mathfrak{h}, \mathfrak{g}} (\mathfrak{e}_U) \xrightarrow{\iota} \mathrm{Aut}_\mathfrak{h} (\mathfrak{e}_U) \xrightarrow{\tau} C_\psi \xrightarrow{\mathcal{W}} H^2_{\mathrm{RBLie}} (\mathfrak{g}_T, \mathfrak{h}_S).
\end{align*}
\end{theorem}

%********************************************************

%In the following, we will consider the Wells exact sequence for abelian extensions. We first observe the following interesting result in an abelian extension.

In the next result, we observe that the group $\mathrm{Aut}_\mathfrak{h}^{\mathfrak{h}, \mathfrak{g}} (\mathfrak{e}_U)$ can be interpreted as the space of derivations on $\mathfrak{g}_T$ with values in $\mathfrak{h}_S$.

\begin{proposition}\label{aut-der}
Let $0 \rightarrow \mathfrak{h}_S \xrightarrow{i} \mathfrak{e}_U \xrightarrow{p} \mathfrak{g}_T \rightarrow 0$ be an abelian extension of the Rota-Baxter Lie algebra $\mathfrak{g}_T$ by a representation $\mathfrak{h}_S$. Then there is a group isomorphism $\mathrm{Aut}_{\mathfrak{h}}^{\mathfrak{h}, \mathfrak{g}} (\mathfrak{e}_U) \cong \mathrm{Der}(\mathfrak{g}_T , \mathfrak{h}_S)$. In particular, the group $\mathrm{Aut}_{\mathfrak{h}}^{\mathfrak{h}, \mathfrak{g}} (\mathfrak{e}_U)$ is abelian.
\end{proposition}

\begin{proof}
Let $\gamma \in \mathrm{Aut}_{\mathfrak{h}}^{\mathfrak{h}, \mathfrak{g}} (\mathfrak{e}_U)$. Then for any $x \in \mathfrak{g},$ we have
\begin{align*}
p ((\gamma s - s) (x)) = p \gamma s (x) - ps (x) = 0.
\end{align*}
Therefore, $(\gamma s - s) (x) \in \mathrm{ker }(p) = \mathrm{im }(i) \cong \mathfrak{h}$. We define a map $\overline{\gamma } : \mathfrak{g} \rightarrow \mathfrak{h}$ by $\overline{\gamma} (x) = (\gamma s -s )(x)$, for $x \in \mathfrak{g}$. Then we have
\begin{align*}
\overline{\gamma} ([x,y]_\mathfrak{g}) =~& \gamma s ([x,y]_\mathfrak{g}) - s ([x, y]_\mathfrak{g} ) \\
=~& \gamma \big(  [s(x), s(y)]_\mathfrak{e} - \chi (x,y)   \big)  - \big(  [s(x), s(y)]_\mathfrak{e} - \chi (x,y) \big) \\
=~& [\gamma s (x), \gamma s (y)]_\mathfrak{e} - [s(x), s(y)]_\mathfrak{e} \qquad (\text{as } \mathrm{im}(\chi) \subset \mathfrak{h} \text{ and } \gamma|_\mathfrak{h} = \mathrm{id}_\mathfrak{h}) \\
=~& [s(x) + \overline{\gamma}(x), s(y) + \overline{\gamma}(y)]_\mathfrak{e} - [s(x), s(y)]_\mathfrak{e}  \\
=~& [s(x) , \overline{\gamma}(y)]_\mathfrak{e} + [\overline{\gamma}(x), s(y)]_\mathfrak{e} \\
=~& \psi_x \overline{\gamma} (y) - \psi_y \overline{\gamma} (x)
\end{align*}
and
\begin{align*}
(S \circ \overline{\gamma} - \overline{\gamma} \circ T)(x) =~& S (\overline{\gamma} (x)) - \big(  \gamma s T(x) - sT(x) \big)  \qquad (\text{as } \overline{\gamma} = \gamma s - s) \\
=~& S (\overline{\gamma} (x)) -  \gamma (Us - \Phi)(x) + (Us - \Phi) (x) \qquad (\text{as } sT = Us - \Phi) \\
=~& S (\overline{\gamma} (x)) - \gamma (Us (x)) + Us (x) \qquad (\text{as } \mathrm{im}(\Phi) \subset \mathfrak{h} \text{ and } \gamma|\mathfrak{h} = \mathrm{id}_\mathfrak{h}) \\
=~& (S \overline{\gamma} - U (\gamma s - s))(x) \qquad (\text{as } U \gamma = \gamma U) \\
=~& (S \overline{\gamma} - U \overline{ \gamma}) (x) = 0 \qquad (\text{as } S = U|_\mathfrak{h}),
\end{align*}
for all $x, y \in \mathfrak{g}$. This shows that $\overline{\gamma} \in \mathrm{Der}(\mathfrak{g}_T, \mathfrak{h}_S)$. Finally, for any $\gamma, \eta \in \mathrm{Aut}_\mathfrak{h}^{\mathfrak{h}, \mathfrak{g}} (\mathfrak{e}_U)$, we have
\begin{align*}
\overline{\gamma \eta} = \gamma \eta s - s =~& \gamma (\eta s - s) + \gamma s - s \\
=~& \gamma \overline{\eta} + \overline{\gamma} = \overline{\eta} + \overline{\gamma} \qquad (\text{as } \mathrm{im}(\overline{\eta}) \subset \mathfrak{h} \text{ and } \gamma|_\mathfrak{h} = \mathrm{id}_\mathfrak{h}).
\end{align*}
This shows that the map ~ $~ \bar{  } ~:~ \mathrm{Aut}_\mathfrak{h}^{\mathfrak{h}, \mathfrak{g}} (\mathfrak{e}_U) \rightarrow \mathrm{Der}(\mathfrak{g}_T, \mathfrak{h}_S)$, $\gamma \mapsto \overline{\gamma}$ is a group homomorphism. 

To show that the map ~ $~ \bar{  }~ : \mathrm{Aut}_\mathfrak{h}^{\mathfrak{h}, \mathfrak{g}} (\mathfrak{e}_U) \rightarrow \mathrm{Der} (\mathfrak{g}_T, \mathfrak{h}_S)$ is injective, take an element $\gamma \in \mathrm{ker }( ~ \bar{  }~ )$. Since $\gamma \in \mathrm{Aut}_\mathfrak{h}^{\mathfrak{h}, \mathfrak{g}} (\mathfrak{e}_U)$, we have $\gamma|_\mathfrak{h} = \mathrm{id}_\mathfrak{h}$ and $p \gamma s = \mathrm{id}_\mathfrak{g}$. Moreover, $\gamma \in \mathrm{ker }(~ \bar{  }~)$ implies that $\overline{\gamma} = 0$. Therefore, for any $e = h + s(x) \in \mathfrak{e}$, we have
\begin{align*}
\gamma (e) = \gamma (h + s(x)) = \gamma (h) + \underbrace{(\gamma s -s)(x)}_{ = \overline{\gamma}(x)} + s(x) = h + s(x) = e
\end{align*}
which implies that $\gamma = \mathrm{id}_\mathfrak{e}$. This shows that the map $~ \bar{  }~$ is injective.
To show that the map $~ \bar{  }~$ is surjective, take any $d \in \mathrm{Der} (\mathfrak{g}_T, \mathfrak{h}_S)$. We define a map $\underline{d} : \mathfrak{e} \rightarrow \mathfrak{e}$ by
\begin{align*}
\underline{d}(h + s(x)) = (h + d(x)) + s(x), \text{ for } e = h+ s(x) \in \mathfrak{e}.
\end{align*}
Then for any elements $e_1 = h_1 + s(x_1)$  and $e_2 = h_2 + s(x_2)$, it is not hard to see that $\underline{d} [e_1, e_2]_\mathfrak{e} = [\underline{d}(e_1), \underline{d}(e_2)]_\mathfrak{e}$. Moreover, we have
\begin{align*}
(U \circ \underline{d})(e) =~& (U \circ \underline{d}) (h+ s(x))\\
=~& U (h + d(x) + s(x)) \\
=~& S (h) + S d(x) + Us (x) \\
=~& S(h) + dT (x) + \Phi (x) + sT (x)  \qquad (\text{as } S d = dT \text{ and } Us = \Phi + sT) \\
=~& \underline{d} \big( S (h) + \Phi (x) + sT (x) \big) \\
=~& \underline{d} \big(  S(h) + Us (x) \big) \\
=~& (\underline{d} \circ U) (h + s(x) ) = (\underline{d} \circ U) (e).
\end{align*}
This shows that $\underline{d} \in \mathrm{Aut} (\mathfrak{e}_U).$ The map $\underline{d}$ induces identity on both $\mathfrak{h}$ and $\mathfrak{g}$ as
\begin{align*}
&\underline{d}|_\mathfrak{h} (h) = \underline{d} (h + s(0) ) = h,  ~ \text{ and } \\
&(p \underline{d} s)(x) = p \underline{d} (0 + s(x)) = p (d(x) + s(x)) = ps(x) ~~~(\text{as } p|_\mathfrak{h} = 0) ~  = x, 
\end{align*}
for $h \in \mathfrak{h}, x \in \mathfrak{g}.$
Hence $\underline{d} \in \mathrm{Aut}_\mathfrak{h}^{\mathfrak{h}, \mathfrak{g}} (\mathfrak{e}_U)$. Finally, it is easy to see that  $~ \bar{  }~ (\underline{d}) = d$. This shows that the map $~ \bar{  }~$ is surjective, and hence bijective. This completes the proof.
\end{proof}

%\subsection{Wells exact sequence}

%\begin{remark}
%\textcolor{red}{LEft: why the above result is not available in non-abelian extension}
%\end{remark}

Combining Theorem \ref{wells-seq-abl} and Proposition \ref{aut-der}, we obtain the following.

\begin{corollary}\label{abl-ext2}
There is a short exact sequence
\begin{align}
\label{wells_cor}
0 \rightarrow \mathrm{Der}(\mathfrak{g}_T, \mathfrak{h}_S) \rightarrow \mathrm{Aut}_\mathfrak{h} (\mathfrak{e}_U) \xrightarrow{\tau} C_\psi \xrightarrow{\mathcal{W}} H^2_{\mathrm{RBLie}} (\mathfrak{g}_T, \mathfrak{h}_S).
\end{align}
\end{corollary}

\medskip

We shall now consider a particular situation where the underlying abelian extension is split. We apply our results to find a description of the automorphism group $\mathrm{Aut}_\mathfrak{h} (\mathfrak{e}_U)$ in terms of compatible pairs and derivations on the Rota-Baxter Lie algebra $\mathfrak{g}_T$. 

An abelian extension $0 \rightarrow \mathfrak{h}_S \xrightarrow{i} \mathfrak{e}_U \xrightarrow{p} \mathfrak{g}_T \rightarrow 0$ (with the $\mathfrak{g}$-module structure on $\mathfrak{h}$ given by $\psi : \mathfrak{g} \rightarrow \mathrm{End}(\mathfrak{h})$) is said to {\bf split} if there exists a section $s : \mathfrak{g} \rightarrow \mathfrak{e}$ which is a morphism of Rota-Baxter Lie algebras from $\mathfrak{g}_T \rightarrow \mathfrak{h}_S$. In this case, we can identify the Rota-Baxter Lie algebra $\mathfrak{e}_U$ as
$\mathfrak{e}_U  \cong  (\mathfrak{g} \oplus \mathfrak{h})_{T \oplus S}$, where the Lie bracket on $\mathfrak{g} \oplus \mathfrak{h}$ is given by the semidirect product
\begin{align}\label{semi-for}
[(x,h), (y,k)]_\ltimes :=  ([x,y]_\mathfrak{g}, \psi_x k - \psi_y h), ~ \text{for } (x,h), (y, k) \in \mathfrak{g} \oplus \mathfrak{h}.
\end{align}
This is precisely the semidirect product Rota-Baxter Lie algebra, denoted by $(\mathfrak{g} \ltimes \mathfrak{h})_{T \oplus S}$. With this identification, the maps $i, p$ and $s$ are the obvious ones. Moreover, since $s: \mathfrak{g}_T \rightarrow \mathfrak{e}_U$ is a Rota-Baxter Lie algebra morphism, we have
\begin{align*}
\chi(x,y) = [s(x), s(y)]_\mathfrak{e} - s[x,y]_\mathfrak{g} = 0 ~~~ \text{ and } ~~~~ \Phi (x) = (Us - sT)(x) = 0, \text{ for } x, y \in \mathfrak{g}.
\end{align*}
Hence $(\chi, \Phi) = 0$ which implies that the Wells map (\ref{w-abel}) vanishes identically. Therefore, the exact sequence (\ref{wells_cor}) takes the following form:
\begin{align}\label{wells-split}
0 \rightarrow \mathrm{Der}(\mathfrak{g}_T, \mathfrak{h}_S) \rightarrow \mathrm{Aut}_\mathfrak{h}((\mathfrak{g} \ltimes \mathfrak{h})_{T \oplus S}) \xrightarrow{\tau} C_\psi \rightarrow 0.
\end{align}

We can check that the above sequence is split exact sequence. To show that, we proceed as follows. For any $(\beta, \alpha) \in C_\psi$, we first define a map $t_{(\beta, \alpha)} : \mathfrak{g} \oplus \mathfrak{h} \rightarrow \mathfrak{g} \oplus \mathfrak{h}$ by $t_{(\beta, \alpha)} (x,h) = (\alpha (x), \beta(h))$, for $(x,h) \in \mathfrak{g} \oplus \mathfrak{h}$. The map $t_{(\beta, \alpha)}$ is a bijection and satisfies $( t_{(\beta, \alpha)} )|_\mathfrak{h} \subset \mathfrak{h}$. Moreover, it is easy to verify that $t_{(\beta, \alpha)}$ preserves the semidirect product bracket (\ref{semi-for}) and commute with the map $T \oplus S : \mathfrak{g} \oplus \mathfrak{h} \rightarrow \mathfrak{g} \oplus \mathfrak{h}$. Hence we have $t_{(\beta, \alpha)} \in \mathrm{Aut}_\mathfrak{h} ((\mathfrak{g} \ltimes \mathfrak{h})_{T \oplus S})$. Therefore, we obtain a map $t : C_\psi \rightarrow \mathrm{Aut}_\mathfrak{h} ((\mathfrak{g} \ltimes \mathfrak{h})_{T \oplus S}), (\beta, \alpha) \mapsto t_{(\beta, \alpha)}.$ The map $t$ obviously satisfies $\tau t = \mathrm{id}_{C_\psi}$ which implies that $t$ is a section of the map $\tau$ in the exact sequence (\ref{wells-split}). Further, the map $t$ is a group homomorphism. Thus, it makes the sequence (\ref{wells-split}) into a split exact sequence. Therefore, we get the following.

\begin{proposition}
Let $0 \rightarrow \mathfrak{h}_S \xrightarrow{i} \mathfrak{e}_U \xrightarrow{p}  \mathfrak{g}_T \rightarrow 0$ be a split and abelian extension of the Rota-Baxter Lie algebra $\mathfrak{g}_T$ by a representation $\mathfrak{h}_S$. Then $\mathrm{Aut}_\mathfrak{h} (\mathfrak{e}_U) \cong C_\psi \ltimes \mathrm{Der}(\mathfrak{g}_T, \mathfrak{h}_S)$ as groups, where we use the semidirect product of groups on the right hand side. In particular, the automorphism group $\mathrm{Aut}_\mathfrak{h} ( (\mathfrak{g} \ltimes \mathfrak{h})_{T \oplus S})$ is itself a semidirect product.
\end{proposition}

\medskip

\medskip

\noindent {\bf Acknowledgements.} A. Das would like to thank IIT Kharagpur (India) for providing a beautiful academic atmosphere where his part of the research has been carried out.

\end{document}